\documentclass[12pt, letterpaper]{amsart}
\usepackage[left=1in,right=1in,bottom=0.5in,top=0.8in]{geometry}
\usepackage{amsfonts}
\usepackage{amsmath, amssymb}
\usepackage{graphicx}
\usepackage[font=small,labelfont=bf]{caption}
\usepackage{epstopdf}
\usepackage{xcolor}
\usepackage{amsthm}
\usepackage{float}
\usepackage{pgfplots}
\usepackage{listings}
\usepackage{longtable}
\usepackage{mathrsfs}
\usepackage{bbold}
\usepackage{comment}
\usepackage{esint}

\usepackage{mathtools}
\usepackage{scalerel}[2014/03/10]
\usepackage[usestackEOL]{stackengine}
\usepackage[utf8]{inputenc}
\usepackage[T1]{fontenc}
\usepackage{enumitem}
\DeclarePairedDelimiterX\set[1]\lbrace\rbrace{#1}

\usepackage{scalerel}[2014/03/10]
\usepackage[usestackEOL]{stackengine}

\usepackage[backref=page]{hyperref}
\hypersetup{
plainpages=false,
colorlinks,
linkcolor={cyan!90!black},
citecolor={magenta},
urlcolor={red!90!black},
bookmarksdepth=2,}

\hypersetup{
pdftitle={Degenerate parabolic equations in divergence form: fundamental solution and Gaussian bounds},
pdfsubject={Mathematics, PDE, Analysis},
pdfauthor={Khalid Baadi},
pdfkeywords={Parabolic systems, Cauchy problems, Green operators, variational methods, Gaussian estimates, Heat Kernel, Energy equality.}
}

\newtheorem{thm}{Theorem}[section]
\newtheorem{cor}[thm]{Corollary}
\newtheorem{prop}[thm]{Proposition}
\newtheorem{lem}[thm]{Lemma}

\theoremstyle{definition}
\newtheorem{defn}[thm]{Definition}

\theoremstyle{remark}
\newtheorem{rem}[thm]{Remark}

\definecolor{energy}{RGB}{114,0,172}
\definecolor{freq}{RGB}{45,177,93}
\definecolor{spin}{RGB}{251,0,29}
\definecolor{signal}{RGB}{203,23,206}
\definecolor{circle}{RGB}{217,86,16}
\definecolor{average}{RGB}{203,23,206}

\definecolor{kb}{rgb}   {.6, 0, 0}

\newcommand{\llangle}{\langle\!\langle}
\newcommand{\rrangle}{\rangle\!\rangle}

\colorlet{shadecolor}{gray!20}
\pgfplotsset{compat=1.9}

\usepgflibrary{fpu}
\makeatletter
\let\c@equation\c@thm
\makeatother
\numberwithin{equation}{section}

\author{Khalid Baadi}
\address{Université Paris-Saclay, CNRS, Laboratoire de Mathématiques d’Orsay, 91405 Orsay, France}
\email{khalid.baadi@universite-paris-saclay.fr}
\date{March 21, 2026}
\title[fundamental solution and Gaussian bounds]{Degenerate parabolic equations in divergence form: fundamental solution and Gaussian bounds}

\DeclareMathOperator*{\supess}{ess\,sup}

\begin{document}

\begin{abstract}
In this paper, we consider second order degenerate parabolic equations with complex, measurable, and time-dependent coefficients. The degenerate ellipticity is dictated by a spatial $A_2$-weight. We prove that having a generalized fundamental solution with upper Gaussian bounds is equivalent to Moser’s $L^2$-$L^\infty$ estimates for local weak solutions. In the special case of real coefficients, Moser's $L^2$-$L^\infty$ estimates are known, which provide an easier proof of Gaussian upper bounds, and a known Harnack inequality is then used to derive Gaussian lower bounds.
\end{abstract}
\maketitle

\tableofcontents

\section{Introduction}\label{section 1}
In this paper, we study parabolic operators of the form 
\begin{equation}\label{H}
    \mathcal{H}u:=  \partial_t u -\omega^{-1} \mathrm{div}_x(A(t,\cdot)\nabla_x u), \ (t,x)\in \mathbb{R}\times \mathbb{R}^n
\end{equation}
where $A=A(t,x)$ is a  matrix-valued function with complex measurable coefficients and the weight $\omega=\omega(x)$ is time-independent and belongs to the spatial Muckenhoupt class $A_2(\mathbb{R}^n,\mathrm{d}x)$. Degeneracy is dictated by the weight $\omega$ in the sense that $\omega^{-1}A$ satisfies the classical uniform ellipticity condition. This operator is defined in Section \ref{section 2}, as well as the meaning of the notation.

Weighted parabolic operators, such as those in \eqref{H}, appear naturally in the analysis of fractional powers of parabolic equations and anomalous diffusion, as shown in \cite{litsgaard2023local} and its references, as well as in the study of heat kernels for Schrödinger equations with singular potentials, as explored in \cite{ishige2017heat}.

The principal purpose of this paper is to establish the following theorem. Rigorous definitions are in Sections \ref{section 2}, \ref{section 3} and \ref{section 4}.

\begin{thm}\label{théorème principal}
The operator $\mathcal{H}=\partial_t  - \omega^{-1}\mathrm{div}_x (A(t,\cdot) \nabla_x )$ has a unique fundamental solution $\Gamma=(\Gamma(t,s))_{t,s \in \mathbb{R}}$. Moreover, the following properties hold.
\begin{enumerate}[label=(\arabic*), leftmargin=*, itemsep=1pt]
    \item \label{Point (1)} The following two statements are equivalent.
    \begin{enumerate}[label=(\roman*)]
        \item The operators $\Gamma(t,s)$, for all $t>s$, have kernels $\Gamma(t,x;s,y)$ with almost everywhere pointwise Gaussian upper bound 
        \begin{equation}\label{BorneGaussienne}
        \left | \Gamma(t,x;s,y) \right | \leq   \frac{K_0}{\sqrt{\omega( B(x,\sqrt{t-s})  )} \, \sqrt{\omega ( B(y,\sqrt{t-s}) )}}  e^{-k_0 \frac{\left | x-y \right |^2}{t-s}}.
        \end{equation}
        The function $\Gamma(t,x;s,y)$ is called a generalized fundamental solution.
        \item All local weak solutions of $\mathcal{H}u=0$ and $\mathcal{H}^\star v=0$ satisfy Moser’s $L^2$-$L^\infty$ estimates. Here, $\mathcal{H}^\star:=-\partial_s-\omega^{-1}\mathrm{div}_x(A^\star(s,\cdot)\nabla_x)$ is the backward operator.
    \end{enumerate}
    \item \label{Point (2)} If $A$ has real-valued coefficients, then  (ii) above is always satisfied, hence \eqref{BorneGaussienne} holds for some constants $K_0 > 0$ and $k_0 > 0$, depending only on the structural constants. Furthermore, there exist constants $\Tilde{K}_0>0$ and $\Tilde{k}_0>0$, depending only on the structural constants, such that the following lower bound holds:
    $$ \frac{\Tilde{K}_0}{\sqrt{\omega( B(x,\sqrt{t-s})  )} \, \sqrt{\omega ( B(y,\sqrt{t-s}) )}}  e^{-\Tilde{k}_0 \frac{\left | x-y \right |^2}{t-s}} \leq  \Gamma(t,x;s,y), $$
    for all $t>s$ and $(x,y)\in \mathbb{R}^n$.
\end{enumerate}
\end{thm}

The existence and uniqueness of the fundamental solution family, proved in \cite{auscher2024fundamental}, are stated in Theorem \ref{thm:FundSol}, along with additional properties. The necessary and sufficient conditions of \ref{Point (1)} are proven in Propositions \ref{ThmConverseHK} and \ref{ThmHkIM}, respectively. Finally, the special case of real-valued coefficients \ref{Point (2)} is addressed in Section \ref{section 5}.

Before delving into the details, it is useful to provide a brief survey of the results that have been obtained so far on this topic of fundamental solution of second order parabolic operators with Gaussian bounds, starting with the unweighted case, \textit{i.e.} $\omega = 1$. When the coefficients are regular, several methods exist for constructing the fundamental solution. The most effective technique combines a parametrix with the freezing point method \cite{friedman2008partial}. This method simplifies the problem by making the coefficients effectively independent of space, leading to explicit solutions represented by smooth kernels $\Gamma(t, x, s, y)$ with Gaussian decay. The treatment of the case of real, merely measurable coefficients was systematically addressed by Aronson \cite{aronson1967bounds, aronson1968non}. He constructed generalized fundamental solutions and proved upper and lower bounds relying on the regularity properties of local weak solutions established by Nash \cite{nash1958continuity} and its extensions, alongside taking limits from operators with regular coefficients. In the case of complex coefficients, Auscher \cite{auscher1996regularity} established Gaussian upper bounds estimates for the fundamental solutions of parabolic equations with complex coefficients, provided that these coefficients are time-independent and are small perturbations of real coefficients. Later, Hofmann and Kim \cite{hofmann2004gaussian} proved the equivalence between Moser’s $L^2$-$L^\infty$ estimates for solutions to parabolic systems and Gaussian upper bounds for the fundamental solution, extending Auscher’s result to include time-dependent coefficients and even systems. 

In the weighted case ($\omega \neq 1$), a work by Cruz-Uribe and Rios \cite{cruz2014corrigendum} establishes the existence of the fundamental solution when $A$ is real-valued, symmetric and independent of $t$. It is given using the semi-group $(e^{-t\mathcal{L}})_{t>0}$, with $\mathcal{L}=-\omega^{-1} \mathrm{div}_x(A(x)\nabla_x )$ and then exploit a Harnack inequality to derive Gaussian bounds. In a recent work, Ataei and Nystr\"om \cite{ataei2024fundamental} extended this result to real-valued and time-dependent coefficients using an approximation argument based on Kato's work \cite{kato1961abstract}, along with the H\"older continuity of weak solutions (\textit{i.e.}, Nash's result for degenerate equations) to construct the fundamental solution and prove pointwise Gaussian upper bounds.

In this paper, having at hand the existence of a fundamental solution family no matter what $A$ is, we can generalize \cite{hofmann2004gaussian} to weighted parabolic equations without assuming smoothness. That is, we prove that having upper bounds of the form \eqref{BorneGaussienne} is equivalent to Moser’s $L^2$-$L^\infty$ for local weak solutions of $\mathcal{H}$ and $\mathcal{H}^\star$. When the coefficients are real, this provides a direct approach, even extending the result of \cite{ataei2024fundamental}, for Gaussian upper bounds, and the argument does not require the use of Nash's regularity result. The latter is indeed valid, and we use it to derive Gaussian lower bounds. For further details, we refer the reader to the main text.

\subsubsection*{\textbf{Notation}}
Throughout this paper, we adopt the following notation.
\begin{enumerate}[label=$\blacklozenge$]
\item We fix an integer $n \ge 1$.
\item For any $(t,x) \in \mathbb{R} \times \mathbb{R}^n$ and $r>0$, we set $Q_r(t,x):=(t-r^2,t]\times B(x,r)$ and $Q_r^\star(t,x):=[t,t+r^2)\times B(x,r)$ where $B(x,r)$ is the Euclidean ball of radius $r$ and center $x$. Thus, $Q_r(t,x)$ and $Q_r^\star(t,x)$ denote the usual forward and backward in time parabolic cylinders.
\item For any $x \in \mathbb{R}^n$ and $r>0$, we set $\omega_r(x):=\omega(B(x,\sqrt{r})):=\int_{B(x,\sqrt{r})}\omega(y) \mathrm{d}y$.
\item We use the notation $\mathcal{D}(\Omega)$ for the space of smooth ($C^\infty$) and compactly supported test functions on an open set $\Omega$. Variables will be indicated at the time of use.
\item We use the sans-serif font "loc" when the prescribed property holds on all compact subsets of the prescribed set.
\item By convention, the notation $C = C(a,b,\dots)$ for a constant means that $ C \in (0, \infty)$ and depends only on $(a,b,\dots)$.
\end{enumerate}

\subsubsection*{\textbf{Acknowledgements}}
I am very grateful to my PhD thesis advisor, Professor Pascal Auscher, for introducing me to the problem, for fruitful discussions, and making useful suggestions to improve a first version of this manuscript.

\section{Preliminaries and basic assumptions}\label{section 2}

\subsection{The weight and function spaces}
The Muckenhoupt class $A_2(\mathbb{R}^n,\mathrm d x)$ is defined as the set of all measurable and positive functions $\omega : \mathbb{R}^n \rightarrow \mathbb{R}$ verifying 
\begin{equation}\label{MuckWeight}
      [ \omega  ]_{A_2}:= \sup_{Q \subset \mathbb{R}^n  } \left ( \frac{1}{|Q|} \int_Q \omega(x) \ \mathrm{d}x   \right ) \left ( \frac{1}{|Q|} \int_Q \omega^{-1}(x) \ \mathrm{d}x   \right ) < \infty,
\end{equation}
where the supremum is taken with respect to all cubes $Q \subset \mathbb{R}^n$ with sides parallel to the axes and $|Q|$ is the Lebesgue measure of $Q$. We refer to \cite[Ch. V]{Stein1993_HA} for general background and for the proofs of all the results concerning weights that we will cite below.

During all this paper, $\omega=\omega(x)$ denotes a fixed weight belonging to the Muckenhoupt class $A_2(\mathbb{R}^n,\mathrm d x)$. By definition \eqref{MuckWeight}, the weight $\omega^{-1}$ is also in the Muckenhoupt class $A_2(\mathbb{R}^n,\mathrm d x)$ with $  [ \omega^{-1}  ]_{A_2}=  [ \omega  ]_{A_2}$. We introduce the measure $\mathrm d \omega := \omega(x) \mathrm d x$ and if $E \subset \mathbb{R}^n$ a Lebesgue measurable set, we write $\omega(E)$ instead of $\int_E \mathrm d \omega$. We recall that $\omega_r(x)=\omega(B(x,\sqrt{r}))$ for any $x \in \mathbb{R}^n$ and $r>0$.

It follows from \eqref{MuckWeight} that there exits constants $\eta \in (0,1)$ and $\beta >0$, depending only on $n$ and $ [ \omega  ]_{A_2}$, such that 
\begin{equation}\label{MuckProportion}
\beta^{-1}\left ( \frac{\left | E \right |}{\left | Q \right |} \right )^{\frac{1}{2\eta }}\leq \frac{\omega(E)}{\omega(Q)}\leq \beta\left ( \frac{\left | E \right |}{\left | Q \right |} \right )^{2\eta},
\end{equation}
whenever $Q \subset \mathbb{R}^n$ is a cube and for all measurable sets $E \subset Q$. In particular, there exists a constant $D$, depending only on $n$ and $ [ \omega  ]_{A_2}$, called the doubling constant for $\omega$ such that
\begin{equation}\label{DoublingMuck}
    \omega(2Q)\leq D \omega(Q),
\end{equation}
for all cubes $Q \subset \mathbb{R}^n$. We may replace cubes by Euclidean balls. For simplicity, we keep using the same notation and constants.

For every $p\ge 1$ and $K \subset \mathbb{R}^n$ a measurable set, we let $L^p_\omega(K)$ be the space of all measurable functions $f:K \rightarrow \mathbb{C}$ such that $$\|f\|_{L^p_\omega(K)}:=\left ( \int_K |f|^p \mathrm d \omega  \right )^{1/p}<\infty.$$
In particular, $L^2_\omega(\mathbb{R}^n)$ is the Hilbert space of square-integrable functions on $\mathbb{R}^n$ with respect to $\mathrm{d}\omega$. We denote its norm by ${\lVert \cdot \rVert}_{2,\omega}$ and its inner product by $\langle \cdot , \cdot \rangle_{2,\omega}$. The class $\mathcal{D}(\mathbb{R}^n)$ is dense in $L^2_{\omega}(\mathbb{R}^n)$ as it is dense in $C_c(\mathbb{R}^n)$, the space of continuous functions on $\mathbb{R}^n$ with compact support, and this latter space is dense in $L^2_{\omega}(\mathbb{R}^n)$ as $\mathrm d \omega$ is a Radon measure on $\mathbb{R}^n$. Moreover, using the $A_2$-condition \eqref{MuckWeight}, we have 
\begin{equation*}
L^2_{\omega}(\mathbb{R}^n) \subset L^1_{\text{loc}}(\mathbb{R}^n, \mathrm{d}x).
\end{equation*}

We define $H^1_{\omega}(\mathbb{R}^n)$ as the space of functions $f \in L^2_\omega(\mathbb{R}^n)$ for which the distributional gradient $\nabla_x f$ belongs to $L^2_\omega(\mathbb{R}^n)^n$, and equip this space with the norm
$\left\| f \right\|_{H^1_\omega} := ( \left\| f \right\|_{2,\omega}^2 + \left\| \nabla_x f \right\|_{2,\omega}^2 )^{1/2}$ making it a Hilbert space. The class $\mathcal{D}(\mathbb{R}^n)$ is dense in $H^1_{\omega}(\mathbb{R}^n)$ and this follows from standard truncation and convolution techniques combined with the boundedness of the maximal operator on $L^2_\omega(\mathbb{R}^n)$. For a proof, see \cite[Thm. 2.5]{kilpelainen1994weighted}.

Finally, we define the measure $\mu$ on $\mathbb{R}^{n+1}$ by $\mathrm{d} \mu(t,x):= \omega(x)\mathrm{d}x\mathrm{d}t.$
\subsection{The distributional duality bracket}
We define $-\Delta_\omega$ as the unbounded self-adjoint operator on $L^2_\omega(\mathbb{R}^n)$ associated to the positive symmetric sesquilinear form on $H^1_\omega(\mathbb{R}^n) \times H^1_\omega(\mathbb{R}^n)$ defined by
\begin{equation*}
    (u,v) \mapsto \int_{\mathbb{R}^n} \nabla_x u \cdot \overline{\nabla_x v}  \ \mathrm d \omega.
\end{equation*}
For all $\beta \geq 0$, we let $(-\Delta_\omega)^\beta$ be the self-adjoint operator $\mathbf{t^{\beta}}(-\Delta_\omega)$ defined by the Borel functional calculus. We refer to \cite{reed1980methods} for more details.

In this paper, \textbf{we adopt the following definition for the distributional duality bracket, which differs from the standard definition. This distinction will not be repeated in subsequent statements, and all equalities in $\mathcal{D'}$ are understood in the sense of this distributional duality bracket.}  
\begin{defn}
    Let $I \subset \mathbb{R}$ be an open set. For $f \in L^{1}_{\mathrm{loc}}(I;L^2_\omega(\mathbb{R}^n)^n)$, $g \in L^{1}_{\mathrm{loc}}(I;L^2_\omega(\mathbb{R}^n))$ and $\beta \in [0,1]$, we define $- \omega^{-1} \mathrm{div}_x(\omega f), (-\Delta_\omega)^{\beta/2}g \in \mathcal{D'}(I\times \mathbb{R}^n)$ by setting for all $\varphi \in \mathcal{D}(I \times \mathbb{R}^n)$,
\begin{align*}
    \llangle - \omega^{-1} \mathrm{div}_x(\omega f) , \varphi \rrangle_{\mathcal{D}',\mathcal{D}}&:=\iint_{I \times \mathbb{R}^n} f(t,x) \cdot \nabla_x \varphi(t,x) \ \mathrm{d}\mu(t,x), \\
    \llangle (-\Delta_\omega)^{\beta/2}g , \varphi \rrangle_{\mathcal{D}',\mathcal{D}}&:=\iint_{I \times \mathbb{R}^n} g(t,x)  ((-\Delta_\omega)^{\beta/2}\varphi(t,\cdot))(x) \ \mathrm{d}\mu(t,x).
\end{align*}
Likewise, if $u \in L^{1}_{\mathrm{loc}}(I;L^2_\omega(\mathbb{R}^n))$, we define its distributional time derivative $\partial_t u \in \mathcal{D'}(I\times \mathbb{R}^n)$ by setting for all $\varphi \in \mathcal{D}(I \times \mathbb{R}^n)$,
\begin{equation*}
    \llangle \partial_t u , \varphi \rrangle_{\mathcal{D}',\mathcal{D}}:=\iint_{I \times \mathbb{R}^n} -u(t,x) \ \partial_t \varphi(t,x) \ \mathrm{d}\mu(t,x).
\end{equation*}
\end{defn}

\subsection{The degenerate parabolic operator}
Throughout this paper, we fix a matrix-valued function $A: \mathbb{R}\times \mathbb{R}^n \rightarrow M_n(\mathbb{C})$ with complex measurable coefficients and such that 
\begin{equation}\label{ellipticité A}
\left | A(t,x)\xi \cdot \zeta  \right |\leq M \omega(x)\left | \xi \right |\left | \zeta  \right |, \ \ \ \  
\nu \left | \xi \right |^2\omega(x)\leq \mathrm{Re}(A(t,x)\xi\cdot \overline{\xi})
\end{equation}
for some $M,\nu>0$ and for all $\zeta, \xi \in \mathbb{C}^n$ and $(t,x) \in \mathbb{R}\times \mathbb{R}^n$.

\begin{defn}[The degenerate parabolic operator]\label{def:duality}
    Let $I\subset \mathbb{R}$ be an open set. For any function $u \in L^1_{\mathrm{loc}}(I;H^1_{\omega,\mathrm{loc}}(\mathbb{R}^n))$ with $\nabla_x u \in L^2(I; L^2_\omega(\mathbb{R}^n))$, we define $\mathcal{H}u=\partial_t u -\omega^{-1} \mathrm{div}_x(A(t,\cdot)\nabla_x u)\in \mathcal{D}'(I \times \mathbb{R}^n)$ by setting for all $\varphi \in  \mathcal{D}(I \times \mathbb{R}^n)$:  
    \begin{align*}
       \llangle \mathcal{H}u, \varphi \rrangle_{\mathcal{D}',\mathcal{D}} &:= \llangle \partial_t u -\omega^{-1} \mathrm{div}_x(A(t,\cdot)\nabla_x u), \varphi \rrangle_{\mathcal{D}',\mathcal{D}}
        \\& = \iint_{I\times \mathbb{R}^n} \left ( -u(t,x)  \partial_t \varphi(t,x) +\omega^{-1}(x) A(t,x)\nabla_x u(t,x) \cdot \nabla_x \varphi(t,x)  \right ) \ \mathrm{d}\mu(t,x).
    \end{align*}
    Likewise, we define $\mathcal{H}^\star$, the formal adjoint of $\mathcal{H}$, by setting $\mathcal{H}^\star u:= - \partial_s u - \omega^{-1} \mathrm{div}_x(A^\star (s,\cdot)\nabla_x u)$ where $A^\star$ is the hermitian adjoint of the matrix $A$.
\end{defn}

\section{Existence, uniqueness, representation and regularity results}\label{section 3}

We apply the method developed in \cite{auscher2024fundamental} in this context by taking $T=\nabla_x:H^1_\omega(\mathbb{R}^n) \to L^2_\omega(\mathbb{R}^n)^n$ which is a closed and densely defined operator on $L^2_\omega(\mathbb{R}^n)$. Moreover, it is injective because the measure $\mathrm{d}\omega$ has infinite mass by \eqref{MuckProportion}, and $\mathcal{D}(\mathbb{R}^n)$ is dense in $D(T) = H^1_\omega(\mathbb{R}^n)$ with respect to the graph norm. Furthermore, we have $T^\star T = -\Delta_\omega$, hence $S=(T^\star T)^{1/2} = (-\Delta_\omega)^{1/2}$.

To connect the degenerate parabolic operator defined above with the one defined in \cite[Section 6]{auscher2024fundamental}, we define, for all $t \in \mathbb{R}$, the sesquilinear form $B_t : H^1_\omega(\mathbb{R}^n) \times H^1_\omega(\mathbb{R}^n) \rightarrow \mathbb{C}$ by setting
\begin{equation*}
   \forall u,v \in H^1_\omega(\mathbb{R}^n) , \  B_t(u,v):=  \int_{\mathbb{R}^{n}} \omega ^{-1}A(t,\cdot)\nabla_xu \cdot \overline{\nabla_x v} \ \mathrm d \omega. 
\end{equation*}
We have for all $u, v \in  H^1_\omega(\mathbb{R}^n)$,
\begin{equation}\label{+}
     \left | B_t(u,v) \right |\leq M \left \| \nabla_x u \right \|_{2,\omega}\left \| \nabla_x v \right \|_{2,\omega}\ ,\ \  \nu \left \| \nabla_x u \right \|_{2,\omega}^2 \leq \mathrm{Re} (B_t(u,u)).
\end{equation}
Moreover, for all $u,v \in H^1_{\omega}(\mathbb{R}^n)$, the function $t \mapsto B_t(u,v)$ is Borel by Fubini's theorem. Therefore, the family $(B_t)_{t\in \mathbb{R}}$ is a weakly measurable family of bounded and coercive sesquilinear forms on $H^1_{\omega}(\mathbb{R}^n) \times H^1_{\omega}(\mathbb{R}^n)$ with respect to the homogeneous norm $\left \| \nabla_x \cdot \right \|_{2,\omega}$ and with uniform bounds $M>0$ and $\nu >0$. 

In this section, we present the key results obtained using the method developed in \cite{auscher2024fundamental} within the context of this paper.
\begin{rem}
   In \cite{auscher2024fundamental}, we have used sesquilinear brackets and we easily bring ourselves to the distributional duality bracket defined above as we have
$$\overline{(-\Delta_\omega)^{\beta/2} \overline{\phi}}=(-\Delta_\omega)^{\beta/2} \phi, \  \text{for all} \ \beta \in [0,1] \ \text{and} \ \phi \in \mathcal{D}(\mathbb{R}^n).$$
\end{rem}
\begin{rem}
    As we use mainly \eqref{+}, our strategy extends to systems as well. 
\end{rem}

\subsection{A Lions' type embedding with integral identities}
We begin by presenting the following proposition, which improves the classical Lions embedding theorem \cite{lions1957problemes}. This result corresponds to \cite[Theorem 2.1]{auscher2024fundamental} in this concrete case.
\begin{prop}\label{prop:energyboundedinterval} 
    Let $I=(0,\mathfrak{T})$ be a  bounded, open interval of $\mathbb{R}$. Let $u \in L^1(I;H^1_\omega(\mathbb{R}^n))$ such that $\nabla_x u\in L^2(I;L^2_\omega(\mathbb{R}^n))$. Assume that $\partial_t u =  - \omega^{-1} \mathrm{div}_x(\omega f)+(-\Delta_\omega)^{\beta/2} g$ in $\mathcal{D'}(I \times \mathbb{R}^n)$ with $f \in L^2(I;L^2_\omega(\mathbb{R}^n)^n)$ and 
     $g \in L^{\rho'}(I;L^2_\omega(\mathbb{R}^n))$, where ${\beta}={2}/{\rho} \in [0,1)$ and $\rho'$ is the conjugate H\"older exponent to $\rho$. Then $ u \in C(\Bar{I},L^2_\omega(\mathbb{R}^n))$
and $ t \mapsto \left \| u(t) \right \|^2_{2,\omega}$ is absolutely continuous on $\Bar{I}$ with, for all $ \sigma, \tau \in \Bar{I}$ such that $  \sigma < \tau$, the integral identity
\begin{align*}
        \left \| u(\tau) \right \|^2_{2,\omega}-\left \| u(\sigma) \right \|^2_{2,\omega} = 2\mathrm{Re}\int_{\sigma}^{\tau} \langle f(t),\nabla_x u(t)\rangle_{{2,\omega}} + \langle g(t),(-\Delta_\omega)^{\beta/2} u(t)\rangle_{{2,\omega}} \  \mathrm d t.
\end{align*}
\end{prop}

\subsection{Fundamental solution}
We define the fundamental solution family as representing the inverse of the degenerate parabolic operator $\mathcal{H}=\partial_t - \omega^{-1}\mathrm{div}_x(A(t,\cdot) \nabla_x)$ on $\mathbb{R}\times \mathbb{R}^n$.
\begin{defn}[Fundamental solution for $\mathcal{H}=\partial_t - \omega^{-1}\mathrm{div}_x(A(t,\cdot) \nabla_x)$]
A fundamental solution for $\mathcal{H}$ is a family $\Gamma=(\Gamma(t,s))_{t,s \in \mathbb{R}}$ of bounded operators on $L^2_\omega(\mathbb{R}^n)$ such that :
\begin{enumerate}
        \item $\sup_{t,s \in \mathbb{R}} \left\|\Gamma(t,s) \right\|_{\mathcal{L}(L^2_\omega(\mathbb{R}^n))} < +\infty.$
        \item $\Gamma(t,s)=0$ if $s>t$.
        \item For all $\psi,\Tilde{\psi}\in \mathcal{D}(\mathbb{R}^n)$, the function $(t,s) \mapsto \langle \Gamma(t,s)\psi, \Tilde{\psi} \rangle_{2,\omega}$ is Borel measurable on $\mathbb{R}^2$.
        \item For all $\phi \in \mathcal{D}(\mathbb{R})$ and $\psi \in \mathcal{D}(\mathbb{R}^n)$, any $u \in L^1_{\mathrm{loc}}(\mathbb{R};H^1_\omega(\mathbb{R}^n))$ with $\int_\mathbb{R}\|\nabla_x u(t)\|^2_{2,\omega}\, \mathrm{d}t<\infty$ solution of the equation $\mathcal{H}u = \phi \otimes \psi $ in $\mathcal{D}'(\mathbb{R}\times\mathbb{R}^{n})$ satisfies $\langle u(t), \Tilde{\psi} \rangle_{2,\omega} = \int_{-\infty}^{t} \phi(s) \langle \Gamma (t,s)\psi, \Tilde{\psi} \rangle_{2,\omega} \ \mathrm{d}s,$ for all $\Tilde{\psi} \in \mathcal{D}(\mathbb{R}^n)$ and for almost every $t\in \mathbb{R}$. We have set $(\phi \otimes \psi)(t,x)=\phi(t) \psi(x)$.
\end{enumerate}
In the same manner, one defines a fundamental solution $(\Tilde{\Gamma}(s,t))_{s,t \in \mathbb{R}}$ to the backward operator \( \mathcal{H}^\star=-\partial_s - \omega^{-1}\mathrm{div}_x ( A^\star(s,\cdot) \nabla_x) \), and (2) is replaced by $\Tilde{\Gamma}(s,t) = 0$ if $s > t$.
\end{defn}
We now present the following theorem, which guarantees the existence of a unique fundamental solution family with several properties.

\begin{thm}\label{thm:FundSol}
The operator $\mathcal{H}$ on $\mathbb{R} \times \mathbb{R}^n$ has a unique fundamental solution, as does $\mathcal{H}^\star$. Moreover, we have the following properties. Functions vanishing at infinity are denoted by $C_0$.
\begin{enumerate}[label=\Alph*)]
    \item (Estimates for the fundamental solution of $\mathcal{H}$) For all $s \in \mathbb{R}$ and $\psi \in L^2_\omega(\mathbb{R}^n)$, $t \mapsto \Gamma(t,s)\psi \in C_0([s,\infty);L^2_\omega(\mathbb{R}^n))$ with $ \Gamma(s,s)\psi =\psi $, $\Gamma(\cdot,s)\psi \in L^1_{\mathrm{loc}}((s,\infty);H^1_\omega(\mathbb{R}^n))$, and there exists a constant $C=C(M,\nu)>0$ such that
    \begin{equation*}
        \sup_{t \geq s} \left \| \Gamma(t,s)  \right \|_{\mathcal{L}(L^2_\omega(\mathbb{R}^n))} \leq C \ \ \mathrm{and} \ \  \int_{s}^{+\infty} \left \| \nabla_x \Gamma(t,s)\psi \right \|^2_{2,\omega}\mathrm d t \leq C^2 \left \| \psi \right \|_{2,\omega}^2.
    \end{equation*}
    Moreover, for any $r \in (2,\infty)$, we have $(-\Delta_\omega)^{\alpha/2}\Gamma(\cdot,s)\psi \in L^r((s,\infty);L^2_\omega(\mathbb{R}^n))$ where $\alpha={2}/{r}$ with
    \begin{equation*}
        \int_{s}^{+\infty}  \|(-\Delta_\omega)^{\alpha/2} \Gamma(t,s)\psi \|^r_{2,\omega}\mathrm d t \leq C^r \left \| \psi \right \|_{2,\omega}^r.
    \end{equation*}
    \item (Estimates for the fundamental solution of $\mathcal{H}^\star$) For all $t \in \mathbb{R}$ and $\psi \in L^2_\omega(\mathbb{R}^n)$, $s \mapsto \Tilde{\Gamma}(s,t)\psi \in C_0((-\infty,t];L^2_\omega(\mathbb{R}^n))$ with $\Tilde{\Gamma}(t,t)\psi =\psi$, $\Tilde{\Gamma}(\cdot,t)\psi \in L^1_{\mathrm{loc}}((-\infty,t);H^1_\omega(\mathbb{R}^n))$, and there exists a constant $C=C(M,\nu)>0$ such that
    \begin{equation*}
       \sup_{s \leq t}  \| \Tilde{\Gamma}(s,t)   \|_{\mathcal{L}(L^2_\omega(\mathbb{R}^n))} \leq C \ \ \mathrm{and} \ \ \int_{-\infty}^{t}  \| \nabla_x \Tilde{\Gamma}(s,t) \psi \|^2_{2,\omega}\mathrm d s \leq C^2 \| \psi  \|_{2,\omega}^2.
    \end{equation*}
    Moreover, for any $r \in (2,\infty)$, we have $(-\Delta_\omega)^{\alpha/2}\Tilde{\Gamma}(\cdot,t)\psi \in L^r((-\infty,t);L^2_\omega(\mathbb{R}^n))$ where $\alpha={2}/{r}$ with
    \begin{equation*}
        \int_{-\infty}^{t}  \|(-\Delta_\omega)^{\alpha/2} \Tilde{\Gamma}(s,t)\psi \|^r_{2,\omega}\mathrm d s \leq C^r \left \| \psi \right \|_{2,\omega}^r.
    \end{equation*}
    \item (Adjointess property) For all $s <t $, $\Gamma(t,s)$ and $\Tilde{\Gamma}(s,t)$ are  adjoint operators.
    \item (Chapman-Kolmogorov identities) For any $s < r <t $, we have $\Gamma(t,s)=\Gamma(t,r)\Gamma(r,s)$.
\end{enumerate}
\end{thm}
\begin{proof}
    The existence and uniqueness follow from Lemma 6.24 and Theorem 6.25 in \cite{auscher2024fundamental}. As for the properties, we refer to Corollary 6.21 and Proposition 6.22 in the same paper.
\end{proof}

\subsection{The Cauchy problem and the fundamental solution}
In this section, we focus on the Cauchy problem on segments and half-lines. We fix $0 < \mathfrak{T} \le \infty$ and consider the model case $(0, \mathfrak{T})$. Let $\rho \in (2, \infty]$, and set $\beta = \frac{2}{\rho} \in [0, 1)$, with $\rho'$ being the H\"older conjugate of $\rho$. The main result of this section is the following proposition.
\begin{prop}[Cauchy problem on $(0,\mathfrak{T})$]\label{prop: Pb de Cauchy borné} Let $ f \in L^2((0,\mathfrak{T});L^2_\omega(\mathbb{R}^n)^n)$, $g\in L^{\rho'}((0,\mathfrak{T});L^2_\omega(\mathbb{R}^n))$ and $\psi \in L^2_\omega(\mathbb{R}^n)$. Then there exists a unique $u \in L^1((0,\mathfrak{T});H^1_\omega(\mathbb{R}^n))$ with $\int_0^\mathfrak{T}\|\nabla_x u(t)\|^2_{2,\omega}\, \mathrm{d}t<\infty$ if $\mathfrak{T}<\infty$ and $u \in L^1_{\mathrm{loc}}((0,\infty);H^1_\omega(\mathbb{R}^n))$ with $\int_0^\infty\|\nabla_x u(t)\|^2_{2,\omega}\, \mathrm{d}t<\infty$ if $\mathfrak{T}=\infty$ solution to the Cauchy problem 
\begin{align*}
\left\{
    \begin{array}{ll}
        \mathcal{H}u  =  - \omega^{-1} \mathrm{div}_x(\omega f) +(-\Delta_\omega)^{\beta/2}g \ \mathrm{in} \  \mathcal{D}'((0,\mathfrak{T})\times \mathbb{R}^n), \\
        u(t) \rightarrow \psi \  \mathrm{ in } \ \mathcal{D'}(\mathbb{R}^n) \ \mathrm{as} \ t \rightarrow 0^+.
    \end{array}
\right.
\end{align*} 
Moreover, $u \in C([0,\mathfrak{T}];L^2_\omega(\mathbb{R}^n))$ with $u(0)=\psi$, $\lim_{t \to \infty} u(t)=0$ if $\mathfrak{T}=\infty$ (by convention, set $u(\infty)=0$), $t \mapsto \| u(t)  \|^2_{2,\omega}$ is absolutely continuous on $[0,\mathfrak{T}]$ and we can write the energy equalities. Furthermore, we have $(-\Delta_\omega)^{\alpha/2} u \in  L^r((0,\mathfrak{T});L^2_\omega(\mathbb{R}^n))$ for any $\alpha \in (0,1]$ with $r=\frac{2}{\alpha} \in [2,\infty)$ with
\begin{align*}
            \sup_{t\in [0,\mathfrak{T}]} \| u(t) \|_{2,\omega}+ \| (-\Delta_\omega)^{\alpha /2} u\|_{L^r((0,\mathfrak{T});L^2_\omega(\mathbb{R}^n))} 
            \leq C  ( \| f  \|_{L^2((0,\mathfrak{T});L^2_\omega(\mathbb{R}^n)^n)}+ \left \| g \right \|_{L^{\rho'}((0,\mathfrak{T});L^2_\omega(\mathbb{R}^n))}+  \| \psi  \|_{2,\omega}  ),
\end{align*} 
where $C=C(M,\nu,\rho)>0$ is a constant. Lastly,  for all $t \in [0,\mathfrak{T}]$, we have the following representation of $u$ (by convention, set $\Gamma(\infty,s)=0$ if $\mathfrak{T}=\infty$): 
\begin{align}\label{rep}
    u(t) = \Gamma(t,0)\psi  + \int_{0}^{t} \Gamma(t,\tau) (- \omega^{-1} \mathrm{div}_x(\omega f))(\tau)\ \mathrm d \tau  + \int_{0}^{t} \Gamma(t,\tau) (-\Delta_\omega)^{\beta/2}g(\tau)\ \mathrm d \tau,
\end{align}
where the two integrals are weakly defined in $L^2_\omega(\mathbb{R}^n)$. More precisely,  for all $\Tilde{\psi}\in L^2_\omega(\mathbb{R}^n)$, we have the equality with absolutely converging integrals
    \begin{align*}
        \langle u(t), \Tilde{\psi} \rangle_{2,\omega} = \langle \Gamma(t,0)\psi, \Tilde{\psi} \rangle_{2,\omega}+ \int_{0}^{t} \langle f(\tau), \nabla_x \Tilde{\Gamma}(\tau,t)\Tilde{\psi} \rangle_{2,\omega}\  \mathrm d \tau + \int_{0}^{t} \langle  g(\tau),(-\Delta_\omega)^{\beta/2}\Tilde{\Gamma}(\tau,t) \Tilde{\psi} \rangle_{2,\omega} \ \mathrm d \tau.
    \end{align*}
When $\rho=\infty$, or equivalently $\beta = 0$, then the last integral in \eqref{rep} converges also strongly in $L^2_\omega(\mathbb{R}^n)$, \textit{i.e.} in the Bochner sense.
\end{prop}
\begin{proof}
The case $\mathfrak{T} = \infty$ follows straightforwardly from combining Theorems 8.1 and 6.35 in \cite{auscher2024fundamental}. For the case $\mathfrak{T} < \infty$, existence follows from restriction from the case $\mathfrak{T} = \infty$, and uniqueness is easily obtained using Proposition \ref{prop:energyboundedinterval} along with \eqref{ellipticité A}.
\end{proof}

\section{Further properties arising from the divergence form}
Having taken the degenerate parabolic operator in divergence form allows us to automatically use cut-off techniques to prove additional properties, as we will see next.

\subsection{\texorpdfstring{$L^2$}{L2}-decay for the fundamental solution}
Given two subsets $E,F \subset \mathbb{R}^n$, we denote by $\mathrm{dist}(E,F)$ the Euclidean distance between these sets. An off-diagonal estimate holds for the fundamental solution as we will see in the next proposition.
\begin{prop}\label{prop:off-diagonal} Let $E,F \subset \mathbb{R}^n$ two measurable sets and let $d:=\mathrm{dist}(E,F)$. Then, there exist two constants $c=c(M,\nu,n)>0$ and $C=C(M,\nu,n)>0$ such that 
\begin{equation*}
    \left\| \Gamma(t,s)f  \right\|_{L^2_\omega(F)} \leq C e^{-c\frac{d^2}{t-s}} \left\|f \right\|_{L^2_\omega(E)},
\end{equation*}
for all $t>s$ and $f \in L^2_\omega(\mathbb{R}^n)$ with  support in $E$. The same statement is true for $\Tilde{\Gamma}(s,t)$.
\end{prop}
\begin{proof}
We follow \cite[Lemma 1]{davies1992heat}. We fix $s\in \mathbb{R}$. Using point A) of Theorem \ref{thm:FundSol}, we can assume that $d>0$. Let $\chi \in C^\infty(\mathbb{R}^n;[0,1])$ such that $\chi=1$ on $E$, $\chi=0$ on $F$ and $\left\|\nabla_x \chi \right\|_{L^\infty(\mathbb{R}^n)} \leq \frac{c}{d}$ with $c=c(n)>0$ is a constant. For all $x \in \mathbb{R}^n$, we set $\phi(x):=e^{\alpha \chi(x)}$ with $\alpha$ a negative constant to fix later. For all $t>s$, we set 
\begin{equation*}
    U(t):=\Gamma(t,s)f.
\end{equation*}
Fix $t>s$, we have $\phi U \in L^2((s,t);H^1_\omega (\mathbb{R}^n))$ and its distributional time derivative verifies
\begin{align*}  
    \partial_t(\phi U) = -\omega^{-1}\left ( A \nabla_x U \cdot \nabla_x \phi \right ) +\omega^{-1} \mathrm{div}_x(A \nabla_x U \phi) \ \ \mathrm{in} \ \mathcal{D'}((s,t)\times \mathbb{R}^n).
\end{align*}
Therefore, by Proposition \ref{prop:energyboundedinterval}, we have $ \tau \mapsto \left \| \phi U(\tau) \right \|^2_{2,\omega}$ is absolutely continuous and we can write the following energy equality:
\begin{align*}
     \|\phi U(t)   \|^2_{2,\omega}-\left \|\phi f  \right \|^2_{2,\omega}&= -2\mathrm{Re}\int_{s}^{t} \int_{\mathbb{R}^n} \omega^{-1}\left (  A(\tau,\cdot)\nabla_x U(\tau)\cdot \nabla_x \phi \right )  \overline{\phi U(\tau)} \ \mathrm d \omega \mathrm d \tau \\&
     \hspace{1.5cm} -2\mathrm{Re}\int_{s}^{t} \int_{\mathbb{R}^n} \omega^{-1}A(\tau,\cdot)\nabla_x U(\tau)\cdot \overline{\nabla_x (\phi U(\tau))} \phi  \ \mathrm d \omega \mathrm d \tau
    \\&= -2\mathrm{Re}\int_{s}^{t} \int_{\mathbb{R}^n} \omega^{-1}A(\tau,\cdot)\nabla_x U(\tau)\cdot \overline{\nabla_x U(\tau)} \phi^2 \ \mathrm d \omega \mathrm d \tau 
    \\ & \hspace{1.5cm}- 4\mathrm{Re} \int_{s}^{t} \int_{\mathbb{R}^n} \omega^{-1}\left ( A(\tau,\cdot)\nabla_x U(\tau)\cdot \nabla_x \phi \right )\overline{U}(\tau) \phi \ \mathrm d \omega \mathrm d \tau.
\end{align*}
Using \eqref{ellipticité A} with the fact that $\nabla_x \phi = \alpha \phi \nabla_x \chi $, we have
\begin{align*} 
       \|\phi U(t)   \|^2_{2,\omega}-\left \| \phi f  \right \|^2_{2,\omega} &\leq -2\nu \int_{s}^{t} \int_{\mathbb{R}^n} 
       | \nabla_x U  |^2 \phi^2 \ \mathrm d \omega \mathrm d \tau + \frac{4M | \alpha |c}{d} \int_{s}^{t} \int_{\mathbb{R}^n}  | \nabla_x U  | \phi \left | U \right |\phi \ \mathrm d \omega \mathrm d \tau
      \\& \leq \frac{2M^2  \alpha^2 c^2}{\nu d^2} \int_{s}^{t} \int_{\mathbb{R}^n} |\phi U  |^2  d \omega \mathrm d \tau.
\end{align*}
Therefore,  
\begin{align*}
  \|\phi U(t)   \|^2_{2,\omega} \leq \left \| \phi f  \right \|^2_{2,\omega} + \frac{2M^2\alpha^2 c^2}{\nu d^2} \int_{s}^{t}   \| \phi U(\tau)  \|^2_{2,\omega} \ \mathrm{d}\tau.
\end{align*}
As this is true for all $t>s$, by Grönwall's lemma, we have
\begin{align*}
       \|\phi U(t)  \|^2_{2,\omega} \leq e^{\frac{2M^2\alpha^2 c^2}{\nu d^2}(t-s)}\left \| \phi f  \right \|^2_{2,\omega}.
\end{align*}
Therefore, as $\phi=1$ on $F$ and $\phi=e^{\alpha}$ on $E$, we get
\begin{align*}
    \left\| \Gamma(t,s)f  \right\|^2_{L^2_\omega(F)} \leq   \|\phi U(t)  \|^2_{2,\omega} \leq e^{\frac{2M^2\alpha^2 c^2}{\nu d^2}(t-s)} \left\| \phi f \right\|^2_{L^2_\omega(E)}=e^{\frac{2M^2\alpha^2 c^2}{\nu d^2}(t-s)+2\alpha} \left\| f \right\|^2_{L^2_\omega(E)}.
\end{align*}
Taking $\alpha := - \frac{\nu}{2 M^2 c^2 (t-s)} d^2$, we get
\begin{align*}
    \left\| \Gamma(t,s)f  \right\|^2_{L^2_\omega(F)} \leq  e^{-\frac{\nu}{2M^2 c^2} \frac{d^2}{t-s}} \left\|f \right\|^2_{L^2_\omega(E)}.
\end{align*}
This property is stable by taking adjoints, so it holds for $\Tilde{\Gamma}(s,t)=\Gamma(t,s)^\star$.
\end{proof}
\begin{cor}\label{cor:L1}
    For all $f \in L^2_\omega(\mathbb{R}^n)$ with compact support, we have $\Gamma(t,s)f \in L^1_\omega(\mathbb{R}^n)$ for all $t\ge s$. The same statement is true for $\Tilde{\Gamma}(s,t)$.
\end{cor}
\begin{proof}
The result is obvious if $t=s$. We assume that $t>s$ and only prove the result for $\Gamma(t,s)$ as the proof is the same for $\Tilde{\Gamma}(s,t)$. Let $R>0$ such that $\mathrm{supp}(f)\subset B(0,R)$. For all $k\ge 0$, we set $C_k(R):=B(0,2^{k+1}R)\setminus B(0,2^{k}R)$. We have 
$$\int_{\mathbb{R}^n}  | \Gamma(t,s)f | \ \mathrm d \omega = \int_{B(0,R)}  | \Gamma(t,s)f | \ \mathrm d \omega + \sum_{k=0}^{\infty} \int_{C_k(R)}  | \Gamma(t,s)f | \ \mathrm d \omega.$$
For all $k \ge 0$, we have $\mathrm{dist}(\mathrm{supp}(f),C_k(R)) \sim 2^k R$. Therefore, by applying the Cauchy–Schwarz inequality together with Proposition \ref{prop:off-diagonal}, we obtain, for all $k\ge0$,
\begin{align*}
    \int_{C_k(R)}  | \Gamma(t,s)f | \ \mathrm d \omega &\le \sqrt{\omega(B(0,2^{k+1}R))} \times \|\Gamma(t,s)f \|_{L^2_\omega(C_k(R))}
    \\& \le C \sqrt{\omega(B(0,R))} \times  e^{-c \frac{4^k R^2}{t-s}} \times D^{\frac{1}{2}(k+1)},
\end{align*}
where we have used the doubling property \eqref{DoublingMuck} in the last line. Therefore, it follows that
\begin{equation*}
    \int_{\mathbb{R}^n}  | \Gamma(t,s)f | \ \mathrm d \omega < \infty.
\end{equation*}
\end{proof}
\subsection{Local weak solutions and Caccioppoli inequality} 
\subsubsection{Local weak solutions}
We say that $u$ is a local weak solution to the equation $\mathcal{H}u=0$ in an open set $\Omega=I \times \mathcal{O} \subset \mathbb{R}^{1+n}$, with $I \subset \mathbb{R}$ and $\mathcal{O} \subset \mathbb{R}^n$ open sets, if $u \in L^1_{\mathrm{loc}}(I; L^2_\omega(\mathcal{O})) $ with $\nabla_x u \in L^2(I; L^2_\omega(\mathcal{O}))$ and satisfies the equation $\mathcal{H}u=\partial_tu-\omega^{-1} \mathrm{div}_x( A(t,\cdot)\nabla_xu)=0$ in $\mathcal{D'}(\Omega)$ as in Definition \ref{def:duality} with $\mathcal{O}$ replacing $\mathbb{R}^n$, namely,
$$\iint_{\Omega} \left ( -u(t,x)  \partial_t \varphi(t,x) +\omega^{-1}(x) A(t,x)\nabla_x u(t,x) \cdot \nabla_x \varphi(t,x)  \right ) \ \mathrm{d}\mu(t,x) =0, \ \text{for all $\varphi \in  \mathcal{D}(\Omega)$}.$$
Local weak solutions to the equation $\mathcal{H}^\star v=0$ are defined similarly.
Whenever $\mathcal{O'}$ is an open set such that $\overline{\mathcal{O'}}\subset O $, then local weak solutions as above are continuous functions of time valued in $L^2_\omega(\mathcal{O'})$. This easily follows from Proposition \ref{prop:energyboundedinterval}. As a result, if $\mathcal{O'}$ is bounded, then we have
\begin{equation}\label{sup pour sol faibles}
        \sup_{t \in I} \left ( \supess_{\mathcal{O'}} \left | u(t,\cdot) \right | \right ) = \supess_{I \times \mathcal{O'}} \left | u \right |,
\end{equation}
This equality is proved in Appendix \ref{annexe}.

\subsubsection{Caccioppoli inequality} 
We present a Caccioppoli inequality for local weak solutions, which will be used later. The proof is included for completeness.
\begin{lem}\label{lem:Caccioppoli}
    Let $(t_0,x_0)\in \mathbb{R}\times \mathbb{R}^n$ and $R>0$. Let $u$ be a local weak solution of $\mathcal{H}u=0$ on a neighborhood (of type $I \times \mathcal{O}$) of $Q_{2R}(t_0,x_0)$. Then, we have the following localized energy inequality:
    \begin{equation*}
       \sup_{t_0-R'^2 \leq t \leq t_0 } \left\|u(t) \right\|^2_{L^2_\omega(B(x_0,R'))} + \int_{Q_{R'}(t_0,x_0)} \left | \nabla_x u \right |^2 \ \mathrm d\mu \leq \frac{C}{(2R-R')^2} \int_{Q_{2R}(t_0,x_0)} \left |  u \right |^2 \ \mathrm d\mu,
    \end{equation*}
    for all $R'\in (0,2R)$, where $C=C(n,M,\nu)>0$ is a constant.
\end{lem}
\begin{proof}
Fix $R'\in (0,2R)$ and let $0 \leq \zeta \leq 1$ be a smooth function such that $\zeta = 1$ on $Q_{R'+\frac{1}{4}(2R-R')}(t_0,x_0)$, $\zeta = 0$ outside $Q_{R'+\frac{1}{2}(2R-R')}(t_0,x_0)$ and verifying $$\left \| \partial_t \zeta  \right \|_{L^\infty(\mathbb{R}^{1+n})}+\left \| \nabla_x \zeta  \right \|^2_{L^\infty(\mathbb{R}^{1+n})} \leq \frac{c(n)}{(2R-R')^2} .$$ 
We set $v := \zeta u $. Then, $v \in L^2((t_0-4R^2,t_0);H^1_\omega(\mathbb{R}^n))$ and satisfies the following equation in $\mathcal{D}'((t_0-4R^2,t_0)\times \mathbb{R}^n)$ :
\begin{equation*}
    \partial_t v -\omega^{-1}\mathrm{div}_x(A(t,\cdot)\nabla_x v) = (\partial_t\zeta) u - \omega^{-1}A(t,\cdot)\nabla_x u \cdot \nabla_x \zeta - \omega^{-1}\mathrm{div}_x(A(t,\cdot) u \nabla_x \zeta).
\end{equation*}
For all $t \in (t_0-4R^2,t_0]$, we set $I_t:=[t_0-4R^2,t]$. Writing the energy equality provided by Proposition \ref{prop:energyboundedinterval}, we have for all $t \in [t_0-R'^2,t_0]$, 
\begin{align*}
    \left \| v(t) \right \|^2_{2,\omega} - \left \| v(t_0-4R^2) \right \|^2_{2,\omega}= 2 \mathrm{Re} &   \int_{I_t} - \langle \omega^{-1}A(s,\cdot) \nabla_x v, \nabla_x v\rangle_{2,\omega}+\langle \partial_t \zeta u , v\rangle_{2,\omega} \\&- \langle \omega^{-1}A(s,\cdot)\nabla_x u \cdot \nabla_x \zeta, v\rangle_{2,\omega}+\langle \omega^{-1}A(s,\cdot)u \nabla_x \zeta, \nabla_x v\rangle_{2,\omega} \  \mathrm d s.
\end{align*}
Since $v(t_0-4R^2)=0$, we have for all $t \in [t_0-R'^2,t_0]$, 
\begin{align*}
    \left \| v(t) \right \|^2_{2,\omega}+ 2\int_{I_t} \mathrm{Re} \ \langle \omega^{-1}A(s,\cdot) \nabla_x v, \nabla_x v\rangle_{2,\omega} \ \mathrm d s = 2\mathrm{Re}  \int_{I_t}   &\langle \partial_t \zeta u , v\rangle_{2,\omega} - \langle \omega^{-1}A(s,\cdot)\nabla_x u \cdot \nabla_x \zeta, v\rangle_{2,\omega} \\& +  \langle \omega^{-1}A(s,\cdot)u \nabla_x \zeta, \nabla_x v\rangle_{2,\omega} \  \mathrm d s.
\end{align*}
Using \eqref{ellipticité A}, we have for all $t \in [t_0-R'^2,t_0]$,
\begin{align*}
    \left \| v(t) \right \|^2_{2,\omega}+ 2 \nu \int_{I_t} \left \| \left ( \nabla_x v \right )(s) \right \|_{2,\omega}^2 \mathrm ds  &\leq \frac{2c(n)}{(2R-R')^2} \int_{Q_{2R}(t_0,x_0)} \left |  u \right |^2 \ \mathrm d\mu + 4M \int_{Q_{2R}(t_0,x_0)} \left |  \nabla_x \zeta \right | \left |   u \right |  \left |  \nabla_x v \right | \ \mathrm d\mu\\&\leq \frac{2c(n)}{(2R-R')^2} \int_{Q_{2R}(t_0,x_0)} \left |  u \right |^2 \ \mathrm d\mu + \frac{4M^2}{\nu} \int_{Q_{2R}(t_0,x_0)} \left |  u \right |^2 \left |  \nabla_x \zeta \right |^2 \ \mathrm d\mu \\& \hspace{5.2cm}+ \nu \int_{Q_{2R}(t_0,x_0)} \left |  \nabla_x v \right |^2 \ \mathrm d\mu.
\end{align*}
Hence, for all $t \in [t_0-R'^2,t_0]$,
\begin{align}\label{Inégalité Cacc}
 \left \| v(t) \right \|^2_{2,\omega}+ 2 \nu \int_{I_t} \left \| \left ( \nabla_x v \right )(s) \right \|_{2,\omega}^2 \mathrm ds \leq \frac{2c(n)+4M^2/\nu }{(2R-R')^2} \int_{Q_{2R}(t_0,x_0)} \left |  u \right |^2 \ \mathrm d\mu + \nu \int_{Q_{2R}(t_0,x_0)} \left | \nabla_x v \right |^2 \ \mathrm d\mu.
\end{align}
When $t=t_0$ and ignoring $\left \| v(t_0) \right \|^2_{2,\omega}$, we have 
\begin{align*}
    2\nu \int_{Q_{2R}(t_0,x_0)} \left | \nabla_x v \right |^2 \ \mathrm d\mu   \leq \frac{ 2c(n)+4M^2/\nu }{(2R-R')^2} \int_{Q_{2R}(t_0,x_0)} \left |  u \right |^2 \ \mathrm d\mu + \nu \int_{Q_{2R}(t_0,x_0)} \left | \nabla_x v \right |^2 \ \mathrm d\mu.
\end{align*}
Therefore,
\begin{align*}
     \nu \int_{Q_{2R}(t_0,x_0)} \left | \nabla_x v \right |^2 \ \mathrm d\mu \leq \frac{2c(n)+4M^2/\nu }{(2R-R')^2} \int_{Q_{2R}(t_0,x_0)} \left |  u \right |^2 \ \mathrm d\mu.
\end{align*}
Since $v = u$ on $Q_{R'}(t_0,x_0)$, we have 
\begin{align}\label{Cacc1}
      \int_{Q_{R'}(t_0,x_0)} \left | \nabla_x u \right |^2 \ \mathrm d\mu \leq \frac{1/\nu\left (2c(n)+4M^2/\nu \right )}{(2R-R')^2} \int_{Q_{2R}(t_0,x_0)} \left |  u \right |^2 \ \mathrm d\mu.
\end{align}
We go back to \eqref{Inégalité Cacc} and this time we ignore the second term of the left hand side. Thus, we have 
$$\sup_{t_0-R'^2 \leq t \leq t_0 } \left\|u(t) \right\|^2_{L^2_\omega(B(x_0,R'))} \leq \frac{2c(n)+4M^2/\nu }{(2R-R')^2} \int_{Q_{2R}(t_0,x_0)} \left |  u \right |^2 \ \mathrm d\mu + \nu \int_{Q_{2R}(t_0,x_0)} \left | \nabla_x v \right |^2 \ \mathrm d\mu.$$
Using the inequality $\left|\nabla_x v \right|^2 \leq 2 \left ( \left|\nabla_x \zeta \right|^2 \left|u \right|^2+ \left|\nabla_x u \right|^2 \left|\zeta \right|^2 \right )$, we have 
\begin{align*}
    \sup_{t_0-R'^2 \leq t \leq t_0 } \left\|u(t) \right\|^2_{L^2_\omega(B(x_0,R'))} \leq \frac{C(n,M,\nu)}{(2R-R')^2} \int_{Q_{2R}(t_0,x_0)} \left |  u \right |^2 \ \mathrm d\mu + 2\nu \int_{Q_{R'+\frac{1}{2}(2R-R')}(t_0,x_0)} \left | \nabla_x u \right |^2 \ \mathrm d\mu.
\end{align*}
Using \eqref{Cacc1}, we have 
$$\int_{Q_{R'+\frac{1}{2}(2R-R')}(t_0,x_0)} \left | \nabla_x u \right |^2 \ \mathrm d\mu \leq  4\frac{C(n,M,\nu)}{(2R-R')^2} \int_{Q_{2R}(t_0,x_0)} \left |  u \right |^2 \ \mathrm d\mu.$$
Therefore, 
\begin{equation}\label{Cacc2}
    \sup_{t_0-R'^2 < t \leq t_0 } \left\|u(t) \right\|^2_{2,\omega} \leq \frac{\Tilde{C}(n,M,\nu) }{(2R-R')^2} \int_{Q_{2R}(t_0,x_0)} \left |  u \right |^2 \ \mathrm d\mu.
\end{equation}
Having established \eqref{Cacc1} and \eqref{Cacc2}, our lemma is therefore proved.
\end{proof}
\begin{rem}[A variant of the Caccioppoli Inequality]\label{rem:varCaccioppoli}
Let $(t_0,x_0)\in \mathbb{R}\times \mathbb{R}^n$, $0 < \rho_1 < \rho_2 < \rho_3 < \rho_4$, $R>0$ and $\psi \in L^2_\omega(\mathbb{R}^n)$ with $\mathrm{supp}(\psi) \subset B(x_0,\rho_1 R)$. We set $\Sigma:=Q_{\rho_3 R}(t_0, x_0) \setminus Q_{\rho_2 R}(t_0, x_0)$ and $\Sigma':=Q_{\rho_4 R}(t_0, x_0) \setminus Q_{\rho_1 R}(t_0, x_0)$. Using an adapted cut-off function $\zeta$ as before with the fact that $\lim_{s \to t_0^{^-}} \| \zeta(s) \Tilde{\Gamma}(s,t_0)\psi \|^2_{2,\omega}=0 $ by Proposition \ref{prop:off-diagonal}, we can similarly prove the following inequality
\begin{equation}\label{Variante Caccioppoli}
    \int_{\Sigma} | \nabla_x (\Tilde{\Gamma}(\cdot,t_0)\psi) |^2 \ \mathrm{d\mu} \leq \frac{C}{\min(\rho_2 - \rho_1, \rho_4 - \rho_3)} \frac{1}{R^2} \int_{\Sigma'} | \Tilde{\Gamma}(\cdot,t_0)\psi |^2 \ \mathrm{d\mu},
\end{equation}
where $C = C(n,M, \nu) > 0$ is a constant.
\end{rem}

\subsection{Conservation property}
In this section, we establish the conservation property, which states that the fundamental solution preserves constants. At this level of generality, the decay estimate shows that these operators map bounded functions to locally square-integrable functions.
\begin{prop}[Conservation property]\label{prop:consproperty}
    For all $ t \ge s$, we have $\Gamma(t,s)1=1$ in $L^2_{\omega,\mathrm{loc}}(\mathbb{R}^n)$, that is, for all $\psi \in L^2_\omega(\mathbb{R}^n)$ with compact support
    $$ \int_{\mathbb{R}^n}(\Tilde{\Gamma}(s,t)\psi)(x) \  \mathrm{d}\omega(x)=\int_{\mathbb{R}^n}\psi(x) \ \mathrm{d}\omega(x). $$
    Similarly, we have $\Tilde{\Gamma}(s,t)1=1$ in $L^2_{\omega,\mathrm{loc}}(\mathbb{R}^n)$.
\end{prop}
\begin{proof}
The proof is the same for both fundamental solutions of $\mathcal{H}$ and $\mathcal{H}^\star$ and we only prove the result for $\Gamma(t,s)$. We fix $\chi \in \mathcal{D}(\mathbb{R}^n)$ such that $\chi=1$ on $B(0,1)$ and $\mathrm{supp}(\chi)\subset B(0,2)$. For all $R>0$ and $x \in \mathbb{R}^n$, we set $\chi_{R}(x):= \chi(\frac{x}{R})$. The result is obvious if $t=s$ and we assume that $t>s$. For all $R>0$, the unique solution $u \in L^1((s,t);H^1_\omega(\mathbb{R}^n))$ with $\int_s^t\|\nabla_x u(t)\|^2_{2,\omega}\, \mathrm{d}t<\infty$ to the  following Cauchy problem : 
\begin{align*}
\left\{
    \begin{array}{ll}
        \partial_t u -\omega^{-1} \mathrm{div}_x(A(t,x)\nabla_x u)  = - \omega^{-1} \mathrm{div}_x(A(t,x) \nabla_x \chi_{R}) \ \mathrm{in} \  \mathcal{D}'((s,t)\times \mathbb{R}^n), \\
        u(\tau) \rightarrow \chi_{R} \  \mathrm{ in } \ \mathcal{D'}(\mathbb{R}^n) \ \mathrm{as} \ \tau \rightarrow s^+
    \end{array}
\right.
\end{align*} 
is the time-constant $u=\chi_{R}$ on $(s,t)$. Using the representation at time $t$ in Proposition \ref{prop: Pb de Cauchy borné}, we have for all $\psi \in L^2_\omega(\mathbb{R}^n)$ with compact support, 
\begin{align*}
    \langle \chi_{R}, \psi \rangle_{2,\omega} &= \langle \Gamma(t,s)\chi_{R}, \psi \rangle_{2,\omega}+ \int_{s}^{t} \langle \omega ^{-1} A(\tau,\cdot) \nabla_x \chi_{R}, \nabla_x \Tilde{\Gamma}(\tau,t)\psi \rangle_{2,\omega}\  \mathrm d \tau\\
    &= \langle \chi_{R}, \Tilde{\Gamma}(s,t)\psi \rangle_{2,\omega}+ \int_{s}^{t} \langle \omega ^{-1} A(\tau,\cdot) \nabla_x \chi_{R}, \nabla_x \Tilde{\Gamma}(\tau,t)\psi \rangle_{2,\omega}\  \mathrm d \tau.
\end{align*}
When $R \to \infty$, the term in the left hand side tends to $\int_{\mathbb{R}^n}\overline{\psi(x)} \ \mathrm{d}\omega(x)$ and the first term in the right hand side tends to $\int_{\mathbb{R}^n}\overline{(\Tilde{\Gamma}(s,t)\psi)(x)} \ \mathrm{d}\omega(x)$ by Corollary \ref{cor:L1}. It remains to show that the second term in the right hand side tends to $0$ as $R \to \infty$. For $R>1$, we have
\begin{align*}
    \int_{s}^{t} |\langle \omega ^{-1} A(\tau,\cdot) \nabla_x \chi_{R}, \nabla_x \Tilde{\Gamma}(\tau,t)\psi \rangle_{2,\omega} | \  \mathrm d \tau &\leq M \int_{s}^{t} \|  \nabla_x \chi_{R} \|_{2,\omega} \| \nabla_x \Tilde{\Gamma}(\tau,t)\psi \|_{L^2_\omega(B(0,2R)\setminus B(0,R))} \  \mathrm d \tau
    \\& \leq \frac{M}{R} (t-s)^{1/2} \omega(B(0,2R))^{1/2}  \|  \nabla_x \chi \|_{L^\infty(\mathbb{R}^n)} \times I_R,
\end{align*}
with $I_R:=\left ( \int_{s}^{t}   \| \nabla_x \Tilde{\Gamma}(\tau,t)\psi \|^2_{L^2_\omega(B(0,2R)\setminus B(0,R))} \  \mathrm d \tau \right )^{1/2}$. Using \eqref{MuckProportion}, we have 
\begin{equation}\label{taille boule}
    \omega(B(0,2R))^{1/2}= \omega(B(0,2))^{1/2} \left ( \frac{\omega(B(0,2R))}{\omega(B(0,2))} \right )^{1/2} \leq \sqrt{\beta}\  \omega(B(0,2))^{1/2} R^{\frac{n}{4\eta}}.
\end{equation}
Using the variant of the Caccioppoli inequality stated in \eqref{Variante Caccioppoli}, for $R>1$ large enough, we have 
$$I_R \leq C  \left(1+\frac{1}{R^2} \right ) \left ( \int_{s-1}^{t}   \|  \Tilde{\Gamma}(\tau,t)\psi \|^2_{L^2_\omega(B(0,3R)\setminus B(0,R/2))} \  \mathrm d \tau \right )^{1/2}. $$
For $R>1$ large enough, we have $\mathrm{dist}(\mathrm{supp}(\psi),B(0,3R)\setminus B(0,R/2)) \sim R$. Thus, using Proposition \ref{prop:off-diagonal}, we get
\begin{equation}\label{IR}
    I_R \leq \Tilde{C} \sqrt{t-s+1} \ \|  \psi \|_{2,\omega} \left(1+\frac{1}{R^2} \right ) e^{-\frac{\Tilde{c}}{t-s+1}R^2}.
\end{equation}
Combining \eqref{IR} and \eqref{taille boule}, we have as desired
$$\int_{s}^{t} |\langle \omega ^{-1} A(\tau,\cdot) \nabla_x \chi_{R}, \nabla_x \Tilde{\Gamma}(\tau,t)\psi \rangle_{2,\omega} | \  \mathrm d \tau \to 0 \ \text{when} \ R \to \infty.$$
\end{proof}

\section{Gaussian upper bounds and Moser’s \texorpdfstring{$L^2$}{L2}-\texorpdfstring{$L^\infty$}{} estimates}\label{section 4}
In this section, we prove the equivalence between Gaussian upper bounds for $\mathcal{H}$ and Moser’s $L^2$-$L^\infty$ estimates for $\mathcal{H}$ and $\mathcal{H}^\star $. We will first define these notions rigorously.

\subsection{Gaussian upper bounds and the generalized fundamental solution}
\begin{defn} We say that $\mathcal{H}$ has Gaussian upper bounds if, for all $t > s$, $\Gamma(t,s)$ is an integral operator with a kernel $\Gamma(t,x;s,y)$ satisfying a pointwise Gaussian upper bound, that is
\begin{equation}\label{pgb}
    \left | \Gamma(t,x;s,y) \right | \leq \frac{K_0}{\sqrt{\omega_{t-s}(x)}\sqrt{\omega_{t-s}(y)}} e^{-k_0 \frac{|x-y|^2}{t-s}},
\end{equation}
for almost every $(x,y) \in \mathbb{R}^{2n}$ and where $K_0 > 0$ and $k_0 > 0$ are constants independent of $t$, $s$, $x$, and $y$. The function $\Gamma(t,x;s,y)$ is referred to as the generalized fundamental solution of $\mathcal{H}$. The definition for $\mathcal{H}^\star$ is analogous, replacing $\Gamma(t,s)$ with $\Tilde{\Gamma}(s,t)$ and $\Gamma(t,x;s,y)$ with $\Tilde{\Gamma}(s,y;t,x)$.
\end{defn}
\begin{lem}\label{lem:Cruz-Uribe et Rios}
The factor $\frac{1}{\sqrt{\omega_{t-s}(x)}\sqrt{\omega_{t-s}(y)}}$ appearing in \eqref{pgb} above may be replaced by one of 
\begin{equation*}
\frac{1}{\omega_{t-s}(x) }, \ \ \frac{1}{\omega_{t-s}(y) }, \ \ \frac{1}{\max(\omega_{t-s}(x) ,\omega_{t-s}(y))},
\end{equation*}
and the constants $K_0$ and $k_0$ in \eqref{pgb} are replaced respectively by $\Tilde{K}_0=\Tilde{K}_0(K_0,k_0,D)>0$ and $\frac{k_0}{2}$.
\end{lem}
\begin{proof}
See \cite[Rem. 3]{cruz2014corrigendum}, where it is proven that there is a constant $C=C(D,k_0)>0$ verifying
\begin{equation*}
    \frac{\sqrt{\omega_{t-s}(y)}}{\sqrt{\omega_{t-s}(x)}} \exp\left( -\frac{k_0}{2} \frac{|x-y|^2}{t-s} \right) \leq C , \ \text{for all $x, y \in \mathbb{R}^n$ and $t > s$.}
\end{equation*}
This uniform bound provides all the required bounds in Lemma \ref{lem:Cruz-Uribe et Rios}.   
\end{proof}
\begin{prop}[Properties of the generalized fundamental solution]\label{prop:solfundgeneralisée}
    $\mathcal{H}$ has Gaussian upper bounds if and only if $\mathcal{H}^\star$ does. In this case, the following properties hold for all $t>s$:
    \begin{enumerate}
        \item (Adjointess property) For almost every $(x,y) \in \mathbb{R}^{2n}$, we have
$$ \Tilde{\Gamma}(s,y;t,x) = \overline{\Gamma(t,x;s,y)}. $$
        \item (Chapman-Kolmogorov identities) If $t>r>s$, then for almost every $(x,y)\in \mathbb{R}^{2n}$, we have
        $$ \int_{\mathbb{R}^n} \Gamma(t,x;r,z) \Gamma(r,z;s,y) \ \mathrm{d}\omega(z)= \Gamma(t,x;s,y). $$
        \item (Conservation property) For almost every $x\in \mathbb{R}^n$, we have $$ \int_{\mathbb{R}^n} \Gamma(t,x;s,y) \ \mathrm{d}\omega(y)=1. $$
    \end{enumerate}
\end{prop}
\begin{proof}
Assume that $\mathcal{H}$ has Gaussian upper bounds. We first claim that
\begin{align*}
    &\supess_{x \in \mathbb{R}^n} \left( \int_{\mathbb{R}^n} \left| \Gamma(t,x;s,y) \right| \, \mathrm{d}\omega(y) \right) \leq C(K_0, k_0, D) < \infty,
    \\ &\supess_{x \in \mathbb{R}^n} \left( \int_{\mathbb{R}^n} \int_{\mathbb{R}^n} \left| \Gamma(t,x;r,z) \right| \ \left| \Gamma(r,z;s,y) \right| \, \mathrm{d}\omega(y) \mathrm{d}\omega(z) \right) \leq \Tilde{C}(K_0, k_0, D) < \infty.
\end{align*}
Using this claim, we apply Fubini's theorem to easily derive (1) and (2), respectively from the adjointness property and the Chapman-Kolmogorov identities in Theorem \ref{thm:FundSol} (points C) and D)), and (3) from the conservation property in Proposition \ref{prop:consproperty}. In particular, $\mathcal{H}^\star$ has Gaussian upper bounds by the adjointness property (1). To prove the claim, for almost all $x \in \mathbb{R}^n$, we have, by Lemma \ref{lem:Cruz-Uribe et Rios},
$$ \int_{\mathbb{R}^n} \left| \Gamma(t,x;s,y) \right| \, \mathrm{d}\omega(y) \leq \frac{K_0}{\omega_{t-s}(x)} \int_{\mathbb{R}^n} e^{-k_0 \frac{|x-y|^2}{t-s}} \, \mathrm{d}\omega(y). $$
Setting $C_k := B(x, 2^{k+1} \sqrt{t-s}) \setminus B(x, 2^k \sqrt{t-s})$ for all $k \in \mathbb{N}$, we have
\begin{align*}
    \int_{\mathbb{R}^n} e^{-k_0 \frac{|x-y|^2}{t-s}} \, \mathrm{d}\omega(y) 
    &= \int_{B(x, \sqrt{t-s})} e^{-k_0 \frac{|x-y|^2}{t-s}} \, \mathrm{d}\omega(y)  + \sum_{k=0}^{\infty} \int_{C_k} e^{-k_0 \frac{|x-y|^2}{t-s}} \, \mathrm{d}\omega(y) \\
    &\leq \omega_{t-s}(x) + \sum_{k=0}^{\infty} \omega(B(x, 2^{k+1} \sqrt{t-s})) e^{-k_0 4^k} \\
    &\leq \left( 1 + \sum_{k=0}^{\infty} D^{k+1} e^{-k_0 4^k} \right) \omega_{t-s}(x),
\end{align*}
where we have used the doubling property \eqref{DoublingMuck} in the last inequality. Thus, we have
$$ \int_{\mathbb{R}^n} \left| \Gamma(t,x;s,y) \right| \, \mathrm{d}\omega(y) \leq C(K_0, k_0, D) < \infty, $$
and the first bound of the claim is proved. The second bound of the claim follows easily from this first bound. Finally, if $\mathcal{H}^\star$ has Gaussian upper bounds, then $\mathcal{H}$ does by the adjointness property. 
\end{proof}

\subsection{Moser’s \texorpdfstring{$L^2$}{L2}-\texorpdfstring{$L^\infty$}{Linfini} estimates }

\begin{defn}\label{def:Moser}
    We say that $\mathcal{H}$, respectively $\mathcal{H}^\star$, satisfies Moser’s $L^2$-$L^\infty$ estimates if there exists a constant $B>0$ such that that for all $R>0$, $(t_0,x_0) \in \mathbb{R}^{1+n}$, and all local weak solutions of $\mathcal{H}u=0$ on a neighborhood of $Q_{2R}(t_0,x_0)$, respectively $\mathcal{H}^\star v=0$ and $Q_{2R}^\star(t_0,x_0)$, respectively, have local bounds of the form, respectively,
    \begin{equation}\label{Moser H}
        \supess_{Q_{R}(t_0,x_0)} \left | u \right | \leq B \left ( \frac{1}{\mu(Q_{2R}(t_0,x_0))} \int_{Q_{2R}(t_0,x_0)} \left | u \right |^2 \ \mathrm d\mu\right  )^{1/2},
    \end{equation}
    \begin{equation}\label{Moser Hstar}
        \supess_{Q_{R}^\star(t_0,x_0)} \left | v \right | \leq B \left ( \frac{1}{\mu(Q_{2R}^\star(t_0,x_0))} \int_{Q_{2R}^\star(t_0,x_0)} \left | v \right |^2 \ \mathrm d\mu\right  )^{1/2}.
    \end{equation}
\end{defn}
\begin{rem}
Another notion that can be defined is the local boundedness property: we say that $\mathcal{H}$, respectively $\mathcal{H}^\star$, satisfies the local boundedness property if the equation \eqref{Moser H}, respectively \eqref{Moser Hstar}, are modified respectively as follows:
    \begin{equation}\label{LBP H}
        \supess_{B(x_0, R)} |u(t_0, \cdot)| \leq B \left ( \frac{1}{\mu(Q_{2R}(t_0,x_0))} \int_{Q_{2R}(t_0,x_0)} \left | u \right |^2 \ \mathrm d\mu\right  )^{1/2},
    \end{equation}
    \begin{equation}\label{LBP Hstar}
        \supess_{B(x_0, R)} |v(t_0, \cdot)| \leq B \left ( \frac{1}{\mu(Q_{2R}^\star(t_0,x_0))} \int_{Q_{2R}^\star(t_0,x_0)} \left | v \right |^2 \ \mathrm d\mu\right  )^{1/2}.
    \end{equation}
Up to changing the constant $B$ and the scale, the two notions are equivalent. More precisely, if $\mathcal{H}$ satisfies Moser’s $L^2$-$L^\infty$ estimates then it satisfies the local boundedness property. Conversely, if $\mathcal{H}$ satisfies the local boundedness property then it satisfies Moser’s $L^2$-$L^\infty$ estimates with a scale of $3R$ instead of $2R$ in \eqref{Moser H}. The same statement applies to $\mathcal{H}^\star$. This follows easily from \eqref{sup pour sol faibles} and the doubling property \eqref{DoublingMuck}.
\end{rem}

\begin{rem}
    The conditions \eqref{Moser H} and \eqref{Moser Hstar} are usually presented by taking suprema on $Q_{R}(t_0,x_0)$ and $Q_{R}^\star(t_0,x_0)$ respectively, which means that one would need to know \textit{a priori} that local weak solutions have pointwise values. In contrast, Moser's $L^2$-$L^\infty$ estimates offer a weaker formulation, as they only require taking the essential supremum over $Q_{R}(t_0,x_0)$ and $Q_{R}^\star(t_0,x_0)$. In particular, it does not require knowing that solutions are continuous, or even defined, at every point. This more relaxed formulation, or more precisely the local boundedness property \eqref{LBP H} and \eqref{LBP Hstar}, which is an equivalent formulation, has been used in kinetic and parabolic contexts in \cite{auscher2024Kolmogorovfundamental} and \cite{auscher2023universal}, respectively.
\end{rem}

\subsection{Proof of \ref{Point (1)} of Theorem \ref{théorème principal}}
The reverse implication of \ref{Point (1)} of Theorem \ref{théorème principal} is covered by the following proposition.
\begin{prop}\label{ThmHkIM}
    If $\mathcal{H}$ and $\mathcal{H}^\star$ satisfy Moser's $L^2$-$L^\infty$ estimates, then $\mathcal{H}$ has Gaussian upper bounds. 
\end{prop}
\begin{proof}
We adapt the argument in \cite{hofmann2004gaussian} to include the weight $\omega$. Let $\gamma \geq0$ and $\psi \in \mathrm{Lip}(\mathbb{R}^n)$ a bounded Lipschitz function such that $\left \| \nabla_x \psi \right \|_{L^\infty}(\mathbb{R}^n) \leq \gamma.$ We fix $f \in \mathcal{D}(\mathbb{R}^n)$ and also fix $s \in \mathbb{R}$. We set 
\begin{equation*}
    U(t):=\Gamma(t,s)e^{-\psi}f, \ \ \text{for all} \ t>s.
\end{equation*}
For all $t>s$, we have $e^{\psi}U \in L^2((s,t);H^1_\omega (\mathbb{R}^n))$ and its distributional time derivative verifies
\begin{align*}
    \partial_t(e^{\psi}U) = -\omega^{-1}\left ( A \nabla_x U \cdot \nabla_x \psi \right ) e^{\psi}+\omega^{-1} \mathrm{div}_x(A \nabla_x U e^{\psi}) \ \ \mathrm{in} \ \mathcal{D'}((s,t)\times \mathbb{R}^n).
\end{align*}
Therefore, by Proposition \ref{prop:energyboundedinterval}, we have $ \tau \mapsto \left \| e^{\psi} U(\tau) \right \|^2_{2,\omega}$ is absolutely continuous and we can write the following energy equality:
\begin{align*}
     \|e^{\psi}U(t)   \|^2_{2,\omega}-\left \|f  \right \|^2_{2,\omega}&= -2\mathrm{Re}\int_{s}^{t} \int_{\mathbb{R}^n} \omega^{-1}\left (  A(\tau,\cdot)\nabla_x U(\tau)\cdot \nabla_x \psi \right )  \overline{U(\tau)}   e^{2\psi} \ \mathrm d \omega \mathrm d \tau \\&
     \hspace{1.5cm} -2\mathrm{Re}\int_{s}^{t} \int_{\mathbb{R}^n} \omega^{-1}A(\tau,\cdot)\nabla_x U(\tau)\cdot \overline{\nabla_x (e^\psi U(\tau))}  e^{\psi} \ \mathrm d \omega \mathrm d \tau
    \\&= -2\mathrm{Re}\int_{s}^{t} \int_{\mathbb{R}^n} \omega^{-1}A(\tau,\cdot)\nabla_x U(\tau)\cdot \overline{\nabla_x U(\tau)}e^{2\psi} \ \mathrm d \omega \mathrm d \tau 
    \\ & \hspace{1.5cm}- 4\mathrm{Re} \int_{s}^{t} \int_{\mathbb{R}^n} \omega^{-1}\left ( A(\tau,\cdot)\nabla_x U(\tau)\cdot \nabla_x \psi \right )\overline{U}(\tau) e^{2\psi} \ \mathrm d \omega \mathrm d \tau.
\end{align*}
From \eqref{ellipticité A}, it follows that
\begin{align*} 
       \|e^{\psi}U(t)   \|^2_{2,\omega}-\left \|f  \right \|^2_{2,\omega} &\leq -2\nu \int_{s}^{t} \int_{\mathbb{R}^n} 
       | \nabla_x U  |^2 e^{2\psi} \ \mathrm d \omega \mathrm d \tau + 4M\gamma \int_{s}^{t} \int_{\mathbb{R}^n}  | \nabla_x U  | e^\psi \left | U \right |e^{\psi} \ \mathrm d \omega \mathrm d \tau
      \\& \leq \frac{2M^2 \gamma^2}{\nu} \int_{s}^{t} \int_{\mathbb{R}^n} |e^{\psi} U  |^2  d \omega \mathrm d \tau.
\end{align*}
Therefore, 
\begin{align*}
  \|e^{\psi}U(t)   \|^2_{2,\omega} \leq \left \|f  \right \|^2_{2,\omega} + \frac{2M^2 \gamma^2}{\nu} \int_{s}^{t}   \|e^{\psi}U(\tau)  \|^2_{2,\omega} \ \mathrm{d}\tau.
\end{align*}
As this is true for all $t>s$, by Grönwall's lemma, we have
\begin{align}\label{UUUU}
       \|e^{\psi}U(t)  \|^2_{2,\omega} \leq e^{\frac{2M^2\gamma^2}{\nu}(t-s)}\left \|f  \right \|^2_{2,\omega},
\end{align}
that is to say
\begin{align*}
      \|e^{\psi}\Gamma(t,s)e^{-\psi}f  \|^2_{2,\omega} \leq e^{\frac{2M^2\gamma^2}{\nu}(t-s)}\left \|f  \right \|^2_{2,\omega}.
\end{align*}
This also applies to the adjoint $\Tilde{\Gamma}(s,t)$. In fact, a similar computation (or duality) shows  
\begin{equation*}
    \|e^{-\psi}\Tilde{\Gamma}(s,t)e^{\psi}f  \|^2_{2,\omega} \leq e^{\frac{2M^2\gamma^2}{\nu}(t-s)}\left \|f  \right \|^2_{2,\omega}.
\end{equation*}
Moser’s $L^2$–$L^\infty$ estimate \eqref{Moser H} with $R = \frac{\sqrt{t-s}}{2}$, together with \eqref{sup pour sol faibles}, implies that for all $x \in \mathbb{R}^n$ and for almost every $z \in B(x,R)$,
\begin{align*}
    \left | U(t,z) \right |^2\leq \frac{B^2}{\mu(Q_{\sqrt{t-s}}(t,x))}\int_{s}^{t}
\int_{B(x,\sqrt{t-s})} \left | U(\tau,y) \right |^2\ \mathrm d\mu.
\end{align*}
Hence, for almost every $z \in B(x,R)$,
\begin{align*}
     |e^{\psi(z)} U(t,z) |^2 &\leq \frac{B^2}{(t-s)\omega_{t-s}(x)}\int_{s}^{t}
\int_{B(x,\sqrt{t-s})}e^{2(\psi(z)-\psi(y)}  | e^{\psi(y)}U(\tau,y)  |^2\ \mathrm d\omega(y)  \mathrm d\tau
\\ & \leq \frac{B^2 e^{4\gamma \sqrt{t-s}}}{(t-s)\omega_{t-s}(x)} \int_{s}^{t}
\int_{B(x,\sqrt{t-s})} | e^{\psi(y)}U(\tau,y) |^2\ \mathrm d\omega(y) \mathrm d \tau
\\ & \leq \frac{B^2 e^{4\gamma \sqrt{t-s}}}{(t-s)\omega_{t-s}(x)} \int_{s}^{t} \|e^{\psi}U(\tau)   \|^2_{2,\omega} \ \mathrm d\tau.
\end{align*}
Remark that for all $z \in B(x,R)$, we have $B(z,\frac{\sqrt{t-s}}{2}) \subset B(x,\sqrt{t-s})$, and by using the doubling property \eqref{DoublingMuck}, we obtain
\begin{equation*}
\omega_{t-s}(z)=\omega( B(z,\sqrt{t-s})) \le D \, \omega(B(z,\frac{\sqrt{t-s}}{2})) \le D \omega_{t-s}(x).
\end{equation*}
Thus, for almost every $z \in B(x,R)$,
\begin{align*}
     |e^{\psi(z)} U(t,z) |^2  \leq \frac{DB^2 e^{4\gamma \sqrt{t-s}}}{(t-s)\omega_{t-s}(z)} \int_{s}^{t} \|e^{\psi}U(\tau)   \|^2_{2,\omega} \ \mathrm d\tau.
\end{align*}
Using this together with \eqref{UUUU}, we have for almost every $x \in \mathbb{R}^n$,
\begin{equation}\label{HKim}
     |e^{\psi(x)} U(t,x)  |^2  \leq \frac{DB^2 e^{4\gamma \sqrt{t-s}}}{(t-s)\omega_{t-s}(x)}  \left ( \int_{s}^{t} e^{\frac{2M^2\gamma^2}{\nu}(\tau-s)}\ \mathrm d\tau \right ) \left \|f  \right \|^2_{2,\omega}.
\end{equation}
If $\gamma=0$, we have 
\begin{equation*}
    \left | U(t,x) \right |^2  \leq \frac{DB^2}{\omega_{t-s}(x)} \left \|f  \right \|^2_{2,\omega}.
\end{equation*}
Hence, 
\begin{equation}\label{0Gamma}
   \left \| \sqrt{\omega_{t-s}}\ \Gamma (t,s)f \right \|_{L^\infty(\mathbb{R}^n)}\leq D^{1/2}B  \left \|f  \right \|_{2,\omega}.
\end{equation}
Likewise, we have
\begin{equation}\label{00Gamma}
    \| \sqrt{\omega_{t-s}} \ \Tilde{\Gamma}(s,t)f  \|_{L^\infty(\mathbb{R}^n)}\leq D^{1/2} B  \left \|f  \right \|_{2,\omega}.
\end{equation}
Whereas if $\gamma>0$, it follows from \eqref{HKim} that for almost every $x \in \mathbb{R}^n$,
\begin{equation*}
   |e^{\psi(x)} U(t,x)  |^2  \leq \frac{DB^2 e^{4\gamma \sqrt{t-s}}}{(t-s)\omega_{t-s}(x)}\frac{e^{2M^2\gamma^2(t-s)/\nu}}{2M^2\gamma^2/\nu} \left \|f  \right \|^2_{2,\omega}.
\end{equation*}
Therefore,
\begin{equation*}
     \| \sqrt{\omega_{t-s}}\ e^{\psi} \Gamma(t,s)e^{-\psi}f  \|_{L^\infty(\mathbb{R}^n)}\leq \frac{D^{1/2}B}{M\sqrt{2/\nu}}\frac{e^{2\gamma \sqrt{t-s}+\frac{M^2\gamma^2}{\nu}(t-s)}}{\gamma\sqrt{t-s}}  \left \|f  \right \|_{2,\omega}.
\end{equation*}
Likewise, we have
\begin{equation}\label{AAAAAA}
    \| \sqrt{\omega_{t-s}}\ e^{-\psi} \Tilde{\Gamma}(s,t)e^{\psi}f \|_{L^\infty(\mathbb{R}^n)}\leq \frac{D^{1/2}B}{M\sqrt{2/\nu}}\frac{e^{2\gamma \sqrt{t-s}+\frac{M^2\gamma^2}{\nu}(t-s)}}{\gamma\sqrt{t-s}}  \left \|f  \right \|_{2,\omega}.
\end{equation}
When $\gamma=0$, using \eqref{00Gamma} and a duality argument, we have
\begin{equation}\label{CCCC}
      \|\Gamma(t,s)(\sqrt{\omega_{t-s}}f )  \|_{2,\omega} \leq D^{1/2} B \left \|f  \right \|_{1,\omega},
\end{equation}
and the same holds for $\Tilde{\Gamma}(s,t)$.\\
Now, if $\gamma>0$, by \eqref{AAAAAA} and a duality argument, we have
\begin{equation}\label{BBBBB}
      \|e^{\psi}\Gamma(t,s)(\sqrt{\omega_{t-s}}\ e^{-\psi}f )  \|_{2,\omega} \leq \frac{D^{1/2}B}{M\sqrt{2/\nu}}\frac{e^{2\gamma \sqrt{t-s}+\frac{M^2\gamma^2}{\nu}(t-s)}}{\gamma\sqrt{t-s}}  \left \|f  \right \|_{1,\omega},
\end{equation}
and the same holds for $\Tilde{\Gamma}(s,t)$. Using the Chapman-Kolmogorov identities (point D) in Theorem \ref{thm:FundSol}), we write $\Gamma(t,s)=\Gamma(t,\frac{t+s}{2})\Gamma(\frac{t+s}{2},s).$ Hence, 
\begin{align*}
    \sqrt{\omega_{t-s}}\ &e^\psi \Gamma(t,s)(\sqrt{\omega_{t-s}}\ e^{-\psi}f)=\sqrt{\omega_{t-s} }\ e^{\psi}\Gamma(t,\frac{t+s}{2})e^{-\psi} ( e^{\psi}\Gamma(\frac{t+s}{2},t) \sqrt{\omega_{t-s}}\ e^{-\psi}f )
    \\&=\sqrt{\frac{\omega_{(t-s)}}{\omega_{\frac{t-s}{2}}}}\sqrt{\omega_{(t-s)/2}}\ e^{\psi}\Gamma(t,\frac{t+s}{2})e^{-\psi} \left ( e^{\psi}\Gamma(\frac{t+s}{2},t) \sqrt{\omega_{(t-s)/2}}\ e^{-\psi}\sqrt{\frac{\omega_{(t-s)}}{\omega_{\frac{t-s}{2}}}}f \right ).
\end{align*}
Notice that by the doubling property \eqref{DoublingMuck}, we have 
\begin{equation}\label{Doubling}
    \left \|  \sqrt{\frac{\omega_{(t-s)}}{\omega_{\frac{(t-s)}{2}}}} \right \|_{L^\infty(\mathbb{R}^n)} \leq D^{1/2}.
\end{equation}
Again, when $\gamma=0$, combining \eqref{0Gamma}, \eqref{CCCC} and \eqref{Doubling} gives
\begin{equation}\label{gamma=0}
    \left \| \sqrt{\omega_{t-s}} \ \Gamma(t,s)(\sqrt{\omega_{t-s}} f) \right \|_{L^\infty(\mathbb{R}^n)}\leq B^2 D^2 \left \|f  \right \|_{1,\omega}.
\end{equation}
The same yields for $\Tilde{\Gamma}(s,t)$. Otherwise, if $\gamma>0$, by combining \eqref{AAAAAA}, \eqref{BBBBB} and \eqref{Doubling}, we have
\begin{equation}\label{gammaNon0}
      \| \sqrt{\omega_{t-s}}\ e^\psi \Gamma(t,s)(\sqrt{\omega_{t-s}}\ e^{-\psi}f)  \|_{L^\infty(\mathbb{R}^n)}\leq \frac{B^2D^2\nu}{M^2}\frac{e^{2\gamma \sqrt{2(t-s)}+\frac{M^2\gamma^2}{\nu}(t-s)}}{\gamma^2(t-s)}  \left \|f  \right \|_{1,\omega},
\end{equation}
and the same yields for $\Tilde{\Gamma}(s,t)$.\\
The estimate \eqref{gamma=0} and the Dunford-Pettis theorem \cite{dunford1940linear} ensures that, for all $t>s$, $\Gamma(t,s)$ is an integral operator with a unique kernel $\Gamma(t,x;s,y)$ satisfying, for almost all $x,y \in \mathbb{R}^n$, the estimate
\begin{equation}\label{DDD}
    \left | \Gamma(t,x;s,y) \right |\leq \frac{1}{\sqrt{\omega_{t-s}(x)}\sqrt{\omega_{t-s}(y)}}  B^2 D^2.
\end{equation}
Moreover, we deduce from the estimate \eqref{gammaNon0} that for all $\gamma>0$ and $\psi \in \mathrm{Lip}(\mathbb{R}^n)$ bounded with $\left \| \nabla_x  \psi \right \|_{L^\infty(\mathbb{R}^n)}\leq \gamma$, we have the estimate for almost all $x,y \in \mathbb{R}^n$, 
\begin{align}\label{CCCCC}
    \left | \Gamma(t,x;s,y) \right | \leq \frac{B^2D^2\nu/M^2}{\sqrt{\omega_{t-s}(x)}\sqrt{\omega_{t-s}(y)}} \times \frac{\exp(2\gamma \sqrt{2(t-s)}+\frac{M^2\gamma^2}{\nu}(t-s))}{\gamma^2(t-s)}\times \exp(\psi(y)-\psi(x)).
\end{align}
We fix $t>s$ and $x\neq y\in \mathbb{R}^n$ for which this is valid. We set $\psi(z):=\inf(\gamma\left |z-y  \right |, \gamma \left |x-y  \right |)$ with $\gamma:=\frac{\left |x-y  \right |}{2 \kappa(t-s)}$ where $\kappa:=\frac{M^2}{\nu}$. The function $\psi$ is bounded and $\left \| \nabla_x  \psi \right \|_{L^\infty(\mathbb{R}^n)}\leq \gamma$. The inequality \eqref{CCCCC} becomes 
\begin{equation}\label{DDDD}
     \left | \Gamma(t,x;s,y) \right | \leq \frac{B^2D^2\nu/M^2}{\sqrt{\omega_{t-s}(x)}\sqrt{\omega_{t-s}(y)}} \times \frac{\exp(2\gamma \sqrt{2(t-s)}+\kappa \gamma^2(t-s)-\gamma\left |x-y  \right | )}{\gamma^2(t-s)}.
\end{equation}
To simplify the notation, we set $\xi :=\frac{\left | x-y \right |}{\sqrt{t-s}}$. We rewrite the right-hand side term above as follows 
\begin{align*}
    \frac{\exp(2\gamma \sqrt{2(t-s)}+\kappa \gamma^2(t-s)-\gamma\left |x-y  \right | )}{\gamma^2(t-s)}= 4 \kappa^2 \frac{\exp(\frac{\sqrt{2}}{\kappa}\xi-\frac{1}{4\kappa}\xi^2)}{\xi^2} .
\end{align*}
Combining \eqref{DDDD} and \eqref{DDD} gives
\begin{equation}\label{PGB}
    \left | \Gamma(t,x;s,y) \right |  \leq \frac{B^2D^2}{\sqrt{\omega_{t-s}(x)}\sqrt{\omega_{t-s}(y)}} \min\left ( 1, \frac{4M^2}{\nu} \frac{\exp(\frac{\sqrt{2}}{\kappa}\xi-\frac{1}{4\kappa}\xi^2)}{\xi^2} \right ).
\end{equation}
Finally, choose $R>0$ such for all $\xi'>R$, 
\begin{equation*}
   \frac{4M^2}{\nu} \cdot \frac{\exp(\frac{\sqrt{2}}{\kappa}\xi-\frac{1}{4\kappa}\xi^2)}{\xi^2} \leq \exp(-\frac{{\xi'}^2}{8\kappa})\cdot
\end{equation*}
Hence, if $\xi >R$, \textit{i.e.} $R\sqrt{t-s}<\left | x-y \right |$, it follows from \eqref{PGB} that
\begin{equation*}
     \left | \Gamma(t,x;s,y) \right |  \leq \frac{B^2D^2}{\sqrt{\omega_{t-s}(x)}\sqrt{\omega_{t-s}(y)}} \exp(-\frac{\nu}{8M^2}\frac{{\left | x-y \right |}^2}{(t-s)}).
\end{equation*}
The other case is easy to treat. In fact, if $\frac{\left | x-y \right |}{\sqrt{t-s}}\leq R$ then $\exp(-R^2)\leq \exp\left ( \frac{-\left | x-y \right |^2}{t-s} \right )$ and we use simply \eqref{PGB} to write 
$$ \left | \Gamma(t,x;s,y) \right |  \leq \frac{B^2D^2}{\sqrt{\omega_{t-s}(x)}\sqrt{\omega_{t-s}(y)}} e^{R^2} \exp\left (- \frac{\left | x-y \right |^2}{t-s} \right ).$$
This concludes the proof of Proposition \ref{ThmHkIM}.
\end{proof}

The direct implication of \ref{Point (1)} in Theorem~\ref{théorème principal} is given by the following proposition.

\begin{prop}\label{ThmConverseHK}
    If $\mathcal{H}$ has Gaussian upper bounds, then $\mathcal{H}$ and $\mathcal{H}^\star$ satisfy Moser’s $L^2$-$L^\infty$ estimates.
\end{prop}
\begin{proof}
The proof is the same for both $\mathcal{H}$ and $\mathcal{H}^\star$ and we only prove the result for $\mathcal{H}$. We fix $R>0$, $(t_0,x_0) \in \mathbb{R}^{1+n}$, and let $u$ be a local weak solutions of $\mathcal{H}u=0$ on a neighborhood of $Q_{2R}(t_0,x_0)$.
Using \eqref{sup pour sol faibles}, we need to prove that 
\begin{equation}\label{JJ}
      \sup_{t \in (t_0-R^2, t_0]} \left ( \supess_{B(x_0,R)} \left | u(t,\cdot) \right | \right )   \leq B \left ( \frac{1}{\mu(Q_{2R}(t_0,x_0))} \int_{Q_{2R}(t_0,x_0)} \left | u \right |^2 \ \mathrm d\mu\right  )^{1/2},
\end{equation}
where $B$ is a constant depending only on $K_0,k_0,\nu,M$, $D$ and $n$. To do so, let $\zeta$ be a non negative smooth function such that $\zeta = 1$ on $Q_{\frac{3}{2}R}(t_0,x_0)$, $\zeta = 0$ outside $Q_{\frac{7}{4}R}(t_0,x_0)$ and verifying $$\left \| \partial_t \zeta  \right \|_{L^\infty(\mathbb{R}^{1+n})}+\left \| \nabla_x \zeta  \right \|^2_{L^\infty(\mathbb{R}^{1+n})} \leq \frac{c(n)}{R^2}.$$
Then $v := \zeta u$ is the unique solution $v \in L^2((s,t);H^1_\omega(\mathbb{R}^n))$ to the  following Cauchy problem
\begin{align*}
\left\{
    \begin{array}{ll}
        \mathcal{H} v = (\partial_t \zeta) u - \omega^{-1}A(t,\cdot)\nabla_x u \cdot \nabla_x \zeta - \omega^{-1}\mathrm{div}_x(A(t,\cdot) u \nabla_x \zeta) \ \mathrm{in} \  \mathcal{D}'((t_0-4R^2,t_0) \times \mathbb{R}^n), \\
        v(\tau) \rightarrow 0 \  \mathrm{ in } \ \mathcal{D'}(\mathbb{R}^n) \ \mathrm{as} \ \tau \rightarrow (t_0-4R^2)^+
    \end{array}
\right.
\end{align*} 
By uniqueness and linearity, we write $v=v_1+v_2+v_3$ where $v_k$ is the solution to the above Cauchy problem considering only the $k^{th}$ term in the right-hand side of the first equation. We fix $ t\in (t_0-R^2,t_0]$. We have $v(t)=v_1(t)+v_2(t)+v_3(t)$ in $L^2_\omega(\mathbb{R}^n)$. Hence, for almost every $x \in B(x_0,R)$,
$$u(t,x)=v(t,x)=v_1(t,x)+v_2(t,x)+v_3(t,x).$$
\underline{\textit{Estimate of $\left\| v_1(t,\cdot) \right\|_{L^\infty(B(x_0,R))}$}:} since $(\partial_t \zeta) u$ is compactly supported in time and space, the representation in Proposition \ref{prop: Pb de Cauchy borné} can be written pointwisely as follows : for almost every $x \in B(x_0,R)$, we have
\begin{equation*}
    v_1(t,x)= \int_{t_0-4R^2}^{t}\left ( \Gamma(t,s)(\partial_t \zeta u)(s) \right )(x) \ \mathrm{d}s = \int_{t_0-4R^2}^{t} \int_{\mathbb{R}^n} \Gamma(t,x;s,y)(\partial_t \zeta u)(s,y) \ \mathrm{d} \omega(y) \ \mathrm ds.
\end{equation*}
Since $\partial_t \zeta = 0 $ on $Q_{\frac{3}{2}R}(t_0,x_0)$ and outside $Q_{\frac{7}{4}R}(t_0,x_0)$, we have by Cauchy-Schwarz inequality, 
\begin{equation*}
    \left | v_1(t,x) \right |\leq \left ( \int_{\Sigma} \left | \Gamma(t,x;s,y) \right |^2 \ \mathrm d\mu(s,y) \right )^{1/2}  \left ( \int_{Q_{\frac{7}{4}R}(t_0,x_0) } \left | (\partial_t \zeta u)(s,y) \right |^2 \ \mathrm d\mu(s,y) \right )^{1/2} ,
\end{equation*}
where $\Sigma:= (s<t)\cap(Q_{\frac{7}{4}R}(t_0,x_0) \setminus Q_{\frac{3}{2}R}(t_0,x_0))$. Therefore, for almost every $x \in B(x_0,R)$,
\begin{equation}\label{1 converse HK}
    \left | v_1(t,x) \right |\leq \frac{c(n)}{R^2}\left ( \int_{\Sigma} \left | \Gamma(t,x;s,y) \right |^2 \ \mathrm d\mu(s,y) \right )^{1/2}  \left ( \int_{Q_{2R}(t_0,x_0) } \left | u \right |^2 \ \mathrm d\mu \right )^{1/2}.
\end{equation}
\underline{\textit{Estimate of $\left\| v_2(t,\cdot) \right\|_{L^\infty(B(x_0,R))}$}:} likewise, since $- \omega^{-1}A(t,\cdot)\nabla_x u \cdot \nabla_x \zeta$ is compactly supported in time and space, we have by Cauchy-Schwarz inequality and for almost every $x \in B(x_0,R)$,
\begin{align*}
    \left | v_2(t,x) \right |&\leq \left ( \int_{\Sigma} \left | \Gamma(t,x;s,y) \right |^2 \ \mathrm d\mu(s,y) \right )^{1/2}  \left ( \int_{Q_{\frac{7}{4}R}(t_0,x_0) } \left | \left ( - \omega^{-1}A(t,\cdot)\nabla_x u \cdot \nabla_x \zeta \right )(s,y) \right |^2 \ \mathrm d\mu(s,y) \right )^{1/2}\\
    &\leq \frac{M \times c(n)}{R} \left ( \int_{\Sigma} \left | \Gamma(t,x;s,y) \right |^2 \ \mathrm d\mu(s,y) \right )^{1/2}  \left ( \int_{Q_{\frac{7}{4}R}(t_0,x_0) } \left | \nabla_x u (s,y) \right |^2 \ \mathrm d\mu(s,y) \right )^{1/2}.
\end{align*}
Using Caccioppoli inequality from Lemma \ref{lem:Caccioppoli}, we deduce that for almost every $x \in B(x_0,R)$,
\begin{equation}\label{2 Convers HK}
    \left | v_2(t,x) \right |\leq \frac{C(n,M,\nu)}{R^2} \left ( \int_{\Sigma} \left | \Gamma(t,x;s,y) \right |^2 \ \mathrm d\mu(s,y) \right )^{1/2}  \left ( \int_{Q_{2R}(t_0,x_0) } \left |  u  \right |^2 \ \mathrm d\mu \right )^{1/2}.
\end{equation}
\underline{\textit{Estimate of $\left\| v_3(t,\cdot) \right\|_{L^\infty(B(x_0,R))}$}:} to obtain a similar estimate on $v_3$, we use again the representation in Proposition \ref{prop: Pb de Cauchy borné} to write 
\begin{equation*}
    \langle v_3(t), \phi \rangle_{2,\omega} = \int_{t_0-4R^2}^{t} \langle \omega^{-1} A(s,\cdot)(\nabla_x \zeta)u, \nabla_x \Tilde{\Gamma}(s,t)\phi \rangle_{2,\omega} \ \mathrm{d} s, \ \ \forall \phi \in \mathcal{D}(\mathbb{R}^n).
\end{equation*}
We fix $\phi \in \mathcal{D}(\mathbb{R}^n)$ such that $\mathrm{supp}(\phi)\subset B(x_0,R)$. We write
\begin{align*}
   \langle v_3(t), \phi \rangle_{2,\omega} &= \int_{t_0-4R^2}^{t} \int_{\mathbb{R}^n} \omega^{-1}(y) u(s,y)\hspace{0.05cm} A(s,y) \hspace{0.05cm} \nabla_x \zeta (s,y)   \cdot \overline{\nabla_x \Tilde{\Gamma}(s,t)\phi (y)} \ \mathrm{d} \omega(y) \ \mathrm{d}s\\&= \int_{\Sigma} \omega^{-1} u \hspace{0.05cm}A \hspace{0.05cm} \nabla_x \zeta \cdot \overline{\nabla_x \Tilde{\Gamma}(\cdot,t)\phi} \ \mathrm{d} \mu.
\end{align*}
Hence,
\begin{equation*}
    \left |  \langle v_3(t), \phi \rangle_{2,\omega} \right |\leq \frac{M\times c(n)}{R}  \left ( \int_{Q_{\frac{7}{4}R}(t_0,x_0) } \left |  u  \right |^2 \ \mathrm d\mu \right )^{1/2}  \left ( \int_{\Sigma}  | \nabla_x \Tilde{\Gamma}(s,t)\phi (y) |^2 \ \mathrm d\mu(s,y)\right )^{1/2}.
\end{equation*}
We define $\Sigma':=Q_{4R}(t,x_0)\setminus Q_{\sqrt{\frac{9}{8}}R}(t,x_0)$. We have $\Sigma \subset Q_{3R}(t,x_0)\setminus Q_{\sqrt{\frac{5}{4}}R}(t,x_0) \subset \Sigma'$. Using the variant of the Caccioppoli inequality \eqref{Variante Caccioppoli}, we have 
\begin{equation*}
    \int_{\Sigma}  | \nabla_x \Tilde{\Gamma}(s,t)\phi (y)   |^2 \ \mathrm d\mu(s,y) \leq \frac{C(n,M,\nu)}{R^2} \int_{\Sigma'}   |  \Tilde{\Gamma}(s,t)\phi (y)   |^2 \ \mathrm d\mu(s,y),
\end{equation*}
Hence, there is a constant $\Tilde{C}=\Tilde{C}(n,M,\nu)>0$ such that
\begin{align}\label{QQQQQ}
    \left |  \langle v_3(t), \phi \rangle_{2,\omega} \right |\leq \frac{\Tilde{C}}{R^2}  \left ( \int_{Q_{2R}(t_0,x_0) } \left |  u  \right |^2 \ \mathrm d\mu \right )^{1/2}  \left ( \int_{\Sigma'}  |  \Tilde{\Gamma}(s,t)\phi (y)  |^2 \ \mathrm d\mu(s,y)\right )^{1/2}.
\end{align}
Since $\Gamma(t,s)$ represented by the kernel $\Gamma(t,x;s,y)$, its adjoint $\Tilde{\Gamma}(s,t)$ is also represented by the kernel $\Tilde{\Gamma}(s,y;t,x)= \overline{\Gamma(t,x;s,y)}$ by Proposition \ref{prop:solfundgeneralisée}. We can therefore write 
\begin{align*}
    \int_{\Sigma'} |  \Tilde{\Gamma}(s,t)\phi (y)    |^2 \ \mathrm d\mu(s,y) &= \int_{\Sigma'} \left | \int_{B(x_0,R)}  \overline{\Gamma}(t,z;s,y) \phi(z)\ \mathrm d\omega(z) \right |^2 d \mu(s,y)
    \\ & \leq \supess_{z \in B(x_0,R)} \left ( \int_{\Sigma'}  | \Gamma(t,z;s,y)  |^2 \ \mathrm d\mu(s,y)  \right ) \left \| \phi \right \|^2_{L^1_\omega(B(x_0,R))}.
\end{align*}
Using this in \eqref{QQQQQ}, we have 
\begin{equation*}
    \left |  \langle v_3(t), \phi \rangle_{2,\omega} \right |\leq \frac{\Tilde{C}}{R^2}  \left ( \int_{Q_{2R}(t_0,x_0) } \left |  u  \right |^2 \ \mathrm d\mu \right )^{1/2} \left (\supess_{z \in B(x_0,R)} \int_{\Sigma'}  | \Gamma(t,z;s,y)     |^2 \ \mathrm d\mu(s,y)  \right )^{1/2} \left \| \phi \right \|_{L^1_\omega(B(x_0,R))}.
\end{equation*}
The inequality above is valid for all $L^2$-functions on $B(x_0,R)$. Using the Lebesgue differentiation theorem, we deduce that for almost every $x \in B(x_0,R)$,
\begin{equation}\label{3 Converse HK}
    \left |  v_3(t,x) \right | \leq  \frac{\Tilde{C}}{R^2} \left (\supess_{z \in B(x_0,R)} \int_{\Sigma'} | \Gamma(t,z;s,y)     |^2 \ \mathrm d\mu(s,y)  \right )^{1/2} \left ( \int_{Q_{2R}(t_0,x_0) } \left |  u  \right |^2 \ \mathrm d\mu \right )^{1/2}  .
\end{equation}
Combining \eqref{1 converse HK}, \eqref{2 Convers HK}, \eqref{3 Converse HK}, the assumption \eqref{pgb} on the kernel and Lemma \ref{lem:Cruz-Uribe et Rios}, we deduce that for almost every $x \in B(x_0,R)$
\begin{equation}\label{LL}
    \left |  u(t,x) \right | \leq  \frac{C(n,M,\nu,K_0,k_0)}{R^2} \left (\supess_{z \in B(x_0,R)} \int_{\Sigma'} \frac{1}{\omega_{t-s}(z)^2} e ^{-k_0 \frac{\left | z-y \right |^2}{t-s}} \ \mathrm d\mu(s,y)  \right )^{1/2} \left ( \int_{Q_{2R}(t_0,x_0) } \left |  u  \right |^2 \ \mathrm d\mu \right )^{1/2},
\end{equation}
since $\Sigma \subset \Sigma'$. We claim that the following estimate holds for all $z \in B(x_0,R)$.
\begin{equation}\label{Poids}
    \int_{\Sigma'} \frac{1}{\omega_{t-s}(z)^2} e ^{-k_0 \frac{\left | z-y \right |^2}{t-s}} \ \mathrm d\mu(s,y) \leq C(k_0,D) \frac{R^2}{\omega(B(x_0,2R))}.
\end{equation}
Then the estimate \eqref{JJ} follows from \eqref{LL}.\\
It remains to prove the claim \eqref{Poids}. Using a change of variable in $s$, we have for all $z \in B(x_0,R)$,
\begin{align*}
    \int_{\Sigma'} \frac{1}{\omega_{t-s}(z)^2} e ^{-k_0 \frac{\left | z-y \right |^2}{t-s}} \ \mathrm d\mu(s,y) &= \int_{0}^{\frac{9}{8}R^2}\int_{\sqrt{\frac{9}{8}}R\leq \left |x_0- y \right |<4R} + \int_{\frac{9}{8}R^2}^{16R^2}\int_{B(x_0,4R)} \frac{\exp(-k_0 \frac{\left |z- y \right |^2}{r})}{\omega_r(z)^2} \omega(y)\ \mathrm dy \mathrm d r \\&:= I +II.
\end{align*}
For $y \in B(x_0,4R)\setminus B(x_0,\sqrt{\frac{9}{8}}R)$, we have $\left | z-y \right | \ge \left | x_0-y \right |-\left | z-x_0 \right | \ge (\sqrt{\frac{9}{8}}-1)R$. Thus, we have
\begin{align*}
    I&= \sum_{k=1}^{+\infty} \int_{ \frac{9}{8}\frac{R^2}{4^{k+1}}}^{ \frac{9}{8}\frac{R^2}{4^k}}\int_{\sqrt{\frac{9}{8}}R\leq \left |x_0- y \right |<4R}  \frac{\exp(-k_0 \frac{\left | z-y \right |^2}{r})}{\omega_r(z)^2} \omega(y)\ \mathrm dy \mathrm d r
    \\&  \leq \sum_{k=1}^{+\infty} \int_{ \frac{9}{8}\frac{R^2}{4^{k+1}}}^{ \frac{9}{8}\frac{R^2}{4^k}}\int_{\sqrt{\frac{9}{8}}R\leq \left |x_0- y \right |<4R}  \frac{e^{-k_0 \left ( \sqrt{\frac{9}{8}}-1 \right )^2\frac{R^2}{r}}}{\omega_r(z)^2} \omega(y)\ \mathrm dy \mathrm d r
    \\& \leq \sum_{k=1}^{+\infty} \int_{ \frac{9}{8}\frac{R^2}{4^{k+1}}}^{ \frac{9}{8}\frac{R^2}{4^k}} \frac{1}{\omega_r(z)^2} \ \mathrm dr \ \omega(B(x_0,4R)) e^{-\frac{8k_0}{9}\left ( \sqrt{\frac{9}{8}}-1 \right )^2 4^k} 
    \\& \leq \sum_{k=1}^{+\infty} \ \frac{1}{\omega(B(z,\sqrt{\frac{9}{8}}\frac{R}{2^{k+1}}))^2} \frac{9}{8}R^2  \left ( \frac{1}{4^k}- \frac{1}{4^{k+1}} \right ) e^{-\frac{8k_0}{9}\left ( \sqrt{\frac{9}{8}}-1 \right )^2 4^k} \omega(B(x_0,4R))
\end{align*}
The doubling property \eqref{DoublingMuck} imply that for all $k\ge0$, $\omega(B(z,\sqrt{\frac{9}{8}}R)) \leq D^{k+1}\omega(B(z,\sqrt{\frac{9}{8}}\frac{R}{2^{k+1}}))$. Thus,
\begin{align*}
    I\leq \left (  \sum_{k=1}^{+\infty} D^{2(k+1)} \left ( \frac{1}{4^k}- \frac{1}{4^{k+1}} \right ) e^{-\frac{8k_0}{9}\left ( \sqrt{\frac{9}{8}}-1 \right )^2 4^k}\right )  \frac{\omega(B(x_0,4R))}{\omega(B(z,\sqrt{\frac{9}{8}}R))^2} \frac{9}{8} R^2.
\end{align*}
As $|x_0-z| \leq R$, we use the doubling property \eqref{DoublingMuck} to deduce that
\begin{equation}\label{I}
     I \leq C(k_0,D) \frac{R^2}{\omega(B(x_0,2R))}.
\end{equation}
Estimating $II$ is straightforward by simply writing  
\begin{equation}\label{II}
    II \leq \frac{\omega(B(x_0,R))}{\omega(B(z,\sqrt{\frac{9}{8}}R))^2} \left ( 16R^2- \frac{9}{8}R^2 \right ) \leq C(D) \frac{R^2}{\omega(B(x_0,2R))},
\end{equation}
by the doubling property \eqref{DoublingMuck} as $|x_0-z| \leq R$. The estimate \eqref{Poids} follows from \eqref{I} and \eqref{II}. 
\end{proof}

\section{Case of real-valued coefficients}\label{section 5} 
In this section, we assume that the matrix-valued function $A$ has real-valued coefficients.
\subsection{The Cauchy problem and the generalized fundamental solution}
As the coefficients of $A$ are real-valued, the operators $\Gamma(t,s)$ are nonnegative, as we will see in the next lemma. 
\begin{lem}\label{Gamma est un opérateur positif}
    If $f \in L^2_\omega(\mathbb{R}^n)$ is a real-valued nonnegative function, then $\Gamma(t,s)f$ is also real-valued and nonnegative for all $t \geq s$.
\end{lem}
\begin{proof}
We fix $s \in \mathbb{R}$. The result is obvious if $t=s$ and we only need to treat the case $t>s$. For all $t>s$, we set $U(t):=\Gamma(t,s)f$. For any $\mathfrak{T}>s$, $U,\mathrm{Re}(U) \in L^2((s,\mathfrak{T});H^1_\omega(\mathbb{R}^n))$,
and are both solutions to the Cauchy Problem
\begin{align*}
\left\{
    \begin{array}{ll}
        \mathcal{H}u=0 \ \mathrm{in} \  \mathcal{D}'((s,\mathfrak{T})\times \mathbb{R}^n), \\
        u(\tau) \rightarrow f \  \mathrm{ in } \ \mathcal{D'}(\mathbb{R}^n) \ \mathrm{as} \ \tau \rightarrow s^+.
    \end{array}
\right.
\end{align*} 
By uniqueness in Proposition \ref{prop: Pb de Cauchy borné}, we have $U=\mathrm{Re}(U)$ on $(s,\infty)$. To prove that $U$ is a nonnegative function, we proceed in two steps using an approximation argument.
\newline
\textit{Step 1: regularizing coefficients of $A$:} We follow \cite{ataei2024fundamental}. Let $\theta \in \mathcal{D}(\mathbb{R})$ a nonnegative function with $\int_\mathbb{R} \theta(t)\mathrm{d}t=1$. For all $p\ge1$, let $\theta_p(t)=p \theta(pt)$ be the associated mollifying sequence. We set $A_p(t,x):=(\theta_p \star A(\cdot,x))(t)$, \textit{i.e.}, we mollify the matrix-valued function A in the time variable only. For all $p \geq 1$ and $t \in \mathbb{R}$, we set 
$$B_t^p (u,v):= \int_{\mathbb{R}^n} \omega^{-1}A_p(t,\cdot)\nabla_x u \cdot \overline{ \nabla_x v} \ \mathrm{d}\omega + \frac{1}{p}\langle u, v \rangle_{2,\omega}.$$
We check easily that $\min(1/l, \nu )\left\| u \right\|_{H_\omega^1(\mathbb{R}^n)}^2 \leq \mathrm{Re}(B_t^p(u,u))$ and $\mathrm{Im}(B_t^p(u,u)) \leq \frac{M}{\nu}\mathrm{Re}(B_t^p(u,u))$. In particular, the quadratic form of $\mathrm{Re}(B_t^p(\cdot,\cdot))$ is closed. Moreover, we have 
$$\left| B_t^p(u,u)-B_s^p(u,u) \right| \leq M  \left\| \frac{\mathrm{d} \theta_p}{\mathrm{d} t}\right\|_{L^1(\mathbb{R})} \left| t-s \right| \left\|\nabla_x u  \right\|_{2,\omega}^2\leq \frac{pM \| \dot{\theta} \|_{L^1(\mathbb{R})}}{\nu} \left| t-s \right|\mathrm{Re}(B_t^p(u,u)),$$ 
where $\dot{\theta}$ is the derivative of $\theta$. For all $p\geq 1$, we set $U_p(t):=\Gamma_p(t,s)f$ where $\Gamma_p$ is the fundamental solution of the parabolic operator associated to the family $(B^p_t)_{{t \in \mathbb{R}}}$. Combining \cite[Theorem III]{kato1961abstract} with uniqueness in in $L^2((s,\mathfrak{T});H^1_\omega(\mathbb{R}^n))$ for any $\mathfrak{T}>s$, we have for all $p \geq 1$, $U_p : (s,\infty) \rightarrow L^2_\omega(\mathbb{R}^n)$ is strongly differentiable. Note that $U_p$ is a real-valued function by the same argument as we did for $U$. Since $\nabla_x U_p(t) \in L^2_\omega(\mathbb{R}^n)$, we have $\partial_t \left |U_p(t)  \right |, \nabla_x \left |  U_p(t)\right | \in L^2_\omega(\mathbb{R}^n)$ with
 \begin{align*}
 \partial_t \left |U_p(t)  \right |  =\left\{\begin{matrix}
 \partial_t U_p(t) \ \mathrm{if} \ U_p(t)\geq 0, \\ 
 -\partial_t U_p(t) \ \mathrm{if} \ U_p(t)< 0 ,
 \\ 
 \end{matrix}\right. 
\ \ \ \text{and} \ \ 
\nabla_x \left |  U_p(t)\right |    =\left\{\begin{matrix}
\nabla_x U_p(t)  \ \mathrm{if} \ U_p(t)\geq 0, \\ 
 -\nabla_x U_p(t) \ \mathrm{if} \ U_p(t)< 0 .
 \\ 
 \end{matrix}\right.  
 \end{align*}
Using this, we have
\begin{align*}
      - \frac{\mathrm{d}}{\mathrm{d}t}  \langle U_p(t)-\left | U_p(t) \right | , U_p(t)-\left | U_p(t) \right |\rangle_{2,\omega}
      &= -2 \langle \partial_t \left ( U_p(t)-\left | U_p(t) \right | \right ) , U_p(t)-\left | U_p(t) \right |\rangle_{2,\omega}
      \\&= -  4\langle \partial_t U_p(t) , U_p(t)-\left | U_p(t) \right |\rangle_{2,\omega}
      \\&= 4\int_{\mathbb{R}^n} \omega^{-1} A(t,\cdot)\nabla_x U_p(t) \cdot \nabla_x (U_p(t)-\left | U_p(t) \right |) \ \mathrm{d}\omega \geq 0.
\end{align*}
Integrating from $s$ to $t$ in this inequality, we see that $t \mapsto \| U_p(t)-|U_p(t)|\|^2_{2,\omega}$ is a non-increasing function. Since it vanishes at $t=s$, we have for all $t>s$, $ U_p(t)=\left | U_p(t) \right |$, that is $\Gamma_p(t,s)f= \left | \Gamma_p(t,s)f \right |$, hence $\Gamma_p(t,s)$ is a nonnegative operator.\newline
\textit{Step 2: passing to the limit:} using uniqueness in $L^2((s,\mathfrak{T});H^1_\omega(\mathbb{R}^n))$ for any $\mathfrak{T}>s$ combined with the boundedness of $(U_p)_{p\ge1}$ in $L^2((s,\mathfrak{T});H^1_\omega(\mathbb{R}^n))$ provided by the energy equality, it is easy to check that, up to extracting a sub-sequence, $(U_p)_{p\ge1} $ converges weakly to $U$ when $p \to \infty$ in $L^2((s,\mathfrak{T});L_\omega^2(\mathbb{R}^n))$ for any $\mathfrak{T}>s$, and therefore $U(t)$ is nonnegative for all $t \geq s$.
\end{proof}
Combining Caccioppoli inequality in Lemma \ref{lem:Caccioppoli}, a weighted Sobolev inequality \cite[Theorem 15.26]{heinonen2018nonlinear} and the Moser's iteration principle, we have the following $L^\infty$-estimate on nonnegative local weak solutions. For a proof, one can follow the classical scheme or see \cite[Proposition 2.1]{ishige1999behavior} with lower order coefficients equal to zero.
\begin{lem}\label{lem: Moser cas reel}
    Let $(t_0, x_0) \in \mathbb{R}^{1+n}$ and $R>0$. If $u$ is is a nonnegative local weak solution of $\mathcal{H}u=0$ in a neighborhood of $Q_{2R}(t_0,x_0)$, then 
    \begin{equation*}
        \supess_{Q_{R}(t_0,x_0)} u = \left\|u \right\|_{L^\infty(Q_R(t_0,x_0))}   \leq B \left ( \frac{1}{\mu(Q_{2R}(t_0,x_0))} \int_{Q_{2R}(t_0,x_0)} u ^2 \ \mathrm d \mu\right  )^{1/2}
    \end{equation*}
    where $B=B(n,[ \omega  ]_{A_2},M,\nu)>0$ is a constant. The same estimate holds for nonnegative local weak solution of $\mathcal{H}^\star v=0$.
\end{lem}
By combining Lemma \ref{lem: Moser cas reel} above, Lemma \ref{Gamma est un opérateur positif} and Proposition \ref{ThmHkIM}, we obtain the following result.
\begin{prop}
The operator $\mathcal{H}$ admits a nonnegative generalized fundamental solution $\Gamma(t,x;s,y)$ with, for all $t>s$, almost everywhere pointwise Gaussian upper bound, that is,
    \begin{equation}\label{1}
      0\leq  \Gamma(t,x;s,y)  \leq  \frac{K_0}{\sqrt{\omega_{t-s}(x)}\sqrt{\omega_{t-s}(y)}} e^{-k_0 \frac{\left | x-y \right |^2}{t-s}},
    \end{equation}
for almost every $(x,y) \in \mathbb{R}^{2n}$, where $K_0=K_0(n,[ \omega  ]_{A_2},M,\nu)>0$ and $k_0=k_0(M,\nu)>0$ are constants.
\end{prop}

By using this proposition together with Proposition \ref{prop: Pb de Cauchy borné} with $f = 0$ and $\rho = \infty$, we obtain the following result, which summarizes all the theory developed in the case of real-valued coefficients.

\begin{cor}[Cauchy problem on $(0,\mathfrak{T})$]
Consider $0<\mathfrak{T}\le \infty$, $\psi \in L^2_\omega(\mathbb{R}^n)$ and $g \in L^1((0,\mathfrak{T});L^2_\omega(\mathbb{R}^n))$. Then there exists a unique $u \in L^1((0,\mathfrak{T});H^1_\omega(\mathbb{R}^n))$ with $\int_0^\mathfrak{T}\|\nabla_x u(t)\|^2_{2,\omega}\, \mathrm{d}t<\infty$ if $\mathfrak{T}<\infty$ and $u \in L^1_{\mathrm{loc}}((0,\infty);H^1_\omega(\mathbb{R}^n))$ with $\int_0^\infty\|\nabla_x u(t)\|^2_{2,\omega}\, \mathrm{d}t<\infty$ if $\mathfrak{T}=\infty$ solution to the Cauchy problem 
\begin{align*}
\left\{
    \begin{array}{ll}
        \partial_t u -\omega^{-1} \mathrm{div}_x(A(t,\cdot)\nabla_x u)= g \ \mathrm{in} \  \mathcal{D}'((0,\mathfrak{T})\times \mathbb{R}^n), \\
        u(t,\cdot) \to \psi \ \mathrm{in} \ \mathcal{D}'(\mathbb{R}^n) \ \mathrm{as} \ t \to 0^+.
    \end{array}
\right.
\end{align*} 
Moreover, $u \in C([0,\mathfrak{T}];L^2_\omega(\mathbb{R}^n))$ with $u(0)=\psi$, $\lim_{t \to \infty} u(t)=0$ if $\mathfrak{T}=\infty$ (by convention, set $u(\infty)=0$), $t \mapsto \| u(t)  \|^2_{2,\omega}$ is absolutely continuous on $[0,\mathfrak{T}]$ and we can write the energy equalities. Furthermore, we have $(-\Delta_\omega)^{\alpha/2} u \in  L^r((0,\mathfrak{T});L^2_\omega(\mathbb{R}^n))$ for any $\alpha \in (0,1]$ with $r=\frac{2}{\alpha} \in [2,\infty)$ with
\begin{align*}
            \sup_{t\in [0,\mathfrak{T}]} \| u(t) \|_{2,\omega}+ \| (-\Delta_\omega)^{\alpha /2} u\|_{L^r((0,\mathfrak{T});L^2_\omega(\mathbb{R}^n))} 
            \leq C  (  \left \| g \right \|_{L^{1}((0,\mathfrak{T});L^2_\omega(\mathbb{R}^n))}+  \| \psi  \|_{2,\omega}  ),
\end{align*} 
where $C=C(M,\nu)>0$ is a constant. Lastly,  for all $t \in [0,\mathfrak{T}]$, we have the following representation of $u$ (by convention, set $\Gamma(\infty,x;s,y)=0$ if $\mathfrak{T}=\infty$): 
\begin{equation*}
    u(t,x)= \int_{\mathbb{R}^n} \Gamma(t,x;0,y)\psi(y) \mathrm d \omega(y)+\int_0^t \int_{\mathbb{R}^n} \Gamma(t,x;s,y) g(s,y) \mathrm d \omega(y) \mathrm d s \ \ \text{for a.e } x \in \mathbb{R}^n,
\end{equation*}
where $\Gamma(t,x;s,y)$ is the generalized fundamental solution of $\mathcal{H}$ of Proposition \ref{lem: Moser cas reel}, and which satisfies all the properties stated in Proposition \ref{prop:solfundgeneralisée}.
\end{cor}

\begin{rem}
    Setting $g=0$, the case $\mathfrak{T}<\infty$ reproves \cite[Theorem 1.2]{ataei2024fundamental}, and we even have a larger uniqueness space.
\end{rem}

\subsection{Additional properties using regularity theory for weak solutions}
The Harnack inequality is known in the context of real-valued coefficients. In this section, we derive Gaussian lower bounds using this result.

For all $(t_0,x_0) \in \mathbb{R} \times \mathbb{R}^n$ and $R>0$, we introduce the cylinders
\begin{align*}
    C_R(t_0,x_0)&:=(t_0-R^2,t_0+R^2) \times B(x_0, 2R),\\
    C^+_R(t_0,x_0)&:=(t_0+\frac{1}{4}R^2,t_0+\frac{3}{4}R^2)\times B(x_0,R),\\
    C^-_R(t_0,x_0)&:=(t_0-\frac{3}{4}R^2,t_0-\frac{1}{4}R^2)\times B(x_0,R).
\end{align*}
The Harnack inequality is stated in the following lemma. We refer to \cite[Thm. A]{ishige1999behavior} for the proof.
\begin{lem}\label{lem:Harnack}
    Let $(t_0,x_0)\in \mathbb{R} \times \mathbb{R}^n$ and $R>0$. If $u$ is a nonnegative local weak solution of $\mathcal{H}u=0$ in $C_R(t_0,x_0)$, then 
    \begin{equation*}
        \sup_{C^-_R(t_0,x_0)} u   \leq C \inf_{C^+_R(t_0,x_0)} u,
    \end{equation*}
    where $C=C([ \omega  ]_{A_2},n,M,\nu)>0$ is a constant.
\end{lem}
\begin{cor}
    Let $\mathfrak{T}\in \mathbb{R}$ and $\mathcal{O} \subset \mathbb{R}^n$ an open set. Let $u$ be a nonnegative local weak solution of $\mathcal{H}u=0$ on $\Omega= (\mathfrak{T},\infty)\times \mathcal{O}$ and $\mathcal{O'}\subset \mathcal{O}$ a convex open set with $\delta:=\mathrm{dist}(\mathcal{O'},\partial \mathcal{O})>0$. Then, for all $t>s>\mathfrak{T}$ and $x,y \in \mathcal{O'}$, one has
    \begin{equation*}
        u(s,y) \leq u(t,x)  e^{C\left ( \frac{\left | x-y\right |^2}{t-s}+\frac{4}{\delta^2}(t-s)+\frac{3(t-s)}{2(s-\mathfrak{T})}+1 \right )},
    \end{equation*}
    with $C=C([ \omega  ]_{A_2},n,M,\nu)>0$ is a constant. If $\mathfrak{T}=-\infty$, then for all $t>s$ and $x,y \in \mathcal{O'}$, one has
    \begin{equation}\label{199.5}
        u(s,y) \leq u(t,x)  e^{C\left ( \frac{\left | x-y\right |^2}{t-s}+\frac{4}{\delta^2}(t-s)+1 \right )}.
    \end{equation}
    Finally, if $\mathcal{O}=\mathbb{R}^n$ and $-\infty<\mathfrak{T}$, then for all $t>s>\mathfrak{T}$ and $x,y \in \mathbb{R}^n$, one has
    \begin{equation}\label{200}
        u(s,y) \leq u(t,x)  e^{C ( \frac{\left | x-y\right |^2}{t-s}+\frac{3(t-s)}{2(s-\mathfrak{T})}+1  )}.
    \end{equation}
\end{cor}
\begin{proof}
    We follow \cite[Theorem 5]{aronson1967local}. By Lemma \ref{lem:Harnack}, if $z \in \mathcal{O}$ and $R>0$ with $B(z,2R) \subset \mathcal{O}$, then 
    \begin{equation}\label{Harnack à itérer}
        u(\tau,z) \leq C u(\tau+R^2,z'), \ \text{for all} \ z' \in B(z,R) \ \text{and} \ \tau \in \mathbb{R} \ \text{with} \ \mathfrak{T}<\tau-\frac{3}{2}R^2. 
    \end{equation}
    We fix $t>s>\mathfrak{T}$ and $x, y \in \mathcal{O'}$. Let $N\ge1$ be the integer verifying 
    $$\frac{\left | x-y\right |^2}{t-s}+\frac{4}{\delta^2}(t-s)+\frac{3(t-s)}{2(s-\mathfrak{T})}\le N <  \frac{\left | x-y\right |^2}{t-s}+\frac{4}{\delta^2}(t-s)+\frac{3(t-s)}{2(s-\mathfrak{T})}+1.$$
    We set $R:= \sqrt{\frac{t-s}{N}}$. We have $\frac{\left | x-y\right |}{N}  \le R \le \delta/2$ and $ \mathfrak{T}<s-\frac{3}{2}R^2 $. Now, we connect $(s,y)$ and $(t,x)$ in $(\mathfrak{T},\infty)\times \mathcal{O'}$ by setting, for all $0 \le i \le N$,
    $$ \tau_i:= s+ \frac{i}{N}(t-s) \ \ \text{and} \ \ z_i:= y+\frac{i}{N}(x-y) \in \mathcal{O'} .$$
    Using \eqref{Harnack à itérer}, we have, for all $0 \le i \le N-1$, $u(\tau_i,z_i) \leq C u(\tau_{i+1},z_{i+1})$. Thus, by iterating, we get 
    $$ u(s,y)=u(\tau_0,z_0) \leq C^{N} u(\tau_N,z_N) = e^{N \log(C)} u(t,x) \leq u(t,x) e^{\log(C)\left ( \frac{\left | x-y\right |^2}{t-s}+\frac{4}{\delta^2}(t-s)+\frac{3(t-s)}{2(s-\mathfrak{T})}+1 \right )}. $$
    Finally, the cases where \(\mathfrak{T} = -\infty\) and \(\mathcal{O} = \mathbb{R}^n\) follow by setting \(\mathfrak{T} = -\infty\) and \(\delta = \infty\), respectively.
\end{proof}

\begin{rem}

  Combining the Harnack inequality in Lemma \eqref{lem:Harnack} with the Gaussian bound \eqref{1} and an argument due to Trudinger \cite[Thm. 2.2]{trudinger1968pointwise}, one gets the following estimates on the generalized fundamental solution.
\begin{equation*}
    \left | \Gamma(t,x+h;s,y) -\Gamma(t,x;s,y)\right | \leq \frac{K_0}{\sqrt{\omega_{t-s}(x)}\sqrt{\omega_{t-s}(y)}}\left ( \frac{\left | h \right |}{(t-s)^{1/2}+\left | x-y \right |} \right )^\delta  e^{-k_0 \frac{\left | x-y \right |^2}{t-s}},
\end{equation*}
\begin{equation*}
    \left | \Gamma(t,x;s,y+h) -\Gamma(t,x;s,y)\right | \leq \frac{K_0}{\sqrt{\omega_{t-s}(x)}\sqrt{\omega_{t-s}(y)}}\left ( \frac{\left | h \right |}{(t-s)^{1/2}+\left | x-y \right |} \right )^\delta  e^{-k_0 \frac{\left | x-y \right |^2}{t-s}},
\end{equation*}
for some $\delta>0$ depending only on the structural constants, all $t>s$ and almost every $ x,y,h \in \mathbb{R}^n$ such that $2\left | h \right |\leq (t-s)^{1/2}+\left | x-y \right |$. For proofs and more details, we refer to \cite{ataei2024fundamental} and \cite{trudinger1968pointwise}.
\end{rem}
\begin{rem}\label{rem:Holder régularité}
    If \( u \) is a nonnegative local weak solution to \( \mathcal{H}u = 0 \) on an open set \( \Omega= I \times \mathcal{O}\), then \( u \) is locally Hölder continuous on \( \Omega \). This regularity result is established in \cite[Thm. B]{ishige1999behavior} by disregarding the assumption (A5), as there are no lower-order terms in our case.
\end{rem}

We note the following lemma, which is the first step toward proving Gaussian lower bounds for the generalized fundamental solution.
\begin{lem}\label{lem:AronsonSerrin}
    Let $\mathfrak{T} \in \mathbb{R}$. Let $u$ be a nonnegative local weak solution of $\mathcal{H}u=0$ on $(\mathfrak{T},\infty)\times \mathbb{R}^n$. Let $y \in \mathbb{R}^n$. Assume that there exist \(\gamma > 0\) and a family of positive real numbers \((\alpha_{s,t}(y))_{t>s>\mathfrak{T}}\) verifying
    $$ \mathcal{M}:= \inf_{t>s>\mathfrak{T}} \frac{1}{\alpha_{s,t}(y)} \int_{B(y,\sqrt{\gamma(t-s)})} u(t,x)\ \mathrm{d}\omega(x)\ > 0.$$
    Then, there are two constants $c=c([ \omega  ]_{A_2},n,M,\nu)>0$ and $C=C(\gamma,[ \omega  ]_{A_2},n,M,\nu)>0$ such that
    \begin{equation*}
         C \mathcal{M} \frac{e^{-c (\frac{|x-y|^2}{t-s}+\frac{t-s}{s-\mathfrak{T}}})}{\omega_{\gamma(t-s)}(y)} \alpha_{s,s+\frac{t-s}{2}}(y) \leq u(t,x),
    \end{equation*}
    for all $t>s>\mathfrak{T}$ and all $x \in \mathbb{R}^n$.
\end{lem}
\begin{proof}
    We take inspiration from \cite[Theorem 7']{aronson1967local}. We fix $s,t \in \mathbb{R}$ with $t>s>\mathfrak{T}$. For $z \in B(y,\sqrt{\gamma(t-s)})$, the inequality \eqref{200} implies that 
    \begin{equation}\label{201}
        u(s+\frac{t-s}{2},z) \leq u(s+\frac{2}{3}(t-s),y) e^{C(6\gamma+\frac{1}{6}\frac{t-s}{s-\mathfrak{T}}+1)}.
    \end{equation}
    Using the continuity of $u(s+\frac{1}{2}(t-s),\cdot)$ on $B(y,\sqrt{\frac{\gamma}{2}(t-s)}  )$ (see Remark \ref{rem:Holder régularité}), we can find $\Tilde{z} \in B(y,\sqrt{\frac{\gamma}{2}(t-s)})$ such that 
    \begin{align*}
        u(s+\frac{1}{2}(t-s),\Tilde{z})&= \frac{1}{\omega_{\frac{\gamma}{2}(t-s)}(y)} \int_{B(y,\sqrt{\frac{\gamma}{2}(t-s)})} u(s+\frac{1}{2}(t-s),l) \ \mathrm{d}\omega(l)
        \\&= \frac{1}{\omega_{\frac{\gamma}{2}(t-s)}(y)} \int_{B(y,\sqrt{\gamma((s+\frac{t-s}{2})-s)})} u(s+\frac{1}{2}(t-s),l) \ \mathrm{d}\omega(l).
    \end{align*}
    Thus, by taking $z=\Tilde{z}$ in \eqref{201}, we get
    \begin{equation}\label{202}
       \mathcal{M} \frac{\alpha_{s,s+\frac{t-s}{2}}(y)}{\omega_{\frac{\gamma}{2}(t-s)}(y)} \leq u(s+\frac{2}{3}(t-s),y) e^{C(6\gamma+\frac{1}{6}\frac{t-s}{s-\mathfrak{T}}+1)}.
    \end{equation}
    We fix $x\in \mathbb{R}^n$. Using \eqref{200} again, we have 
    \begin{equation}\label{203}
        u(s+\frac{2}{3}(t-s),y) \leq u(t,x) e^{C(\frac{3|x-y|^2}{t-s}+\frac{1}{3}\frac{t-s}{s-\mathfrak{T}}+1)}.
    \end{equation}
    Combining \eqref{202} and \eqref{203}, we get $$ e^{-C(6\gamma+\frac{1}{6}\frac{t-s}{s-\mathfrak{T}}+1)} \mathcal{M} \frac{e^{-C(\frac{3|x-y|^2}{t-s}+\frac{1}{3}\frac{t-s}{s-\mathfrak{T}}+1)}}{\omega_{\frac{\gamma}{2}(t-s)}(y)} \alpha_{s,s+\frac{1}{2}(t-s)}(y) \leq u(t,x).$$
    Finally, by \eqref{MuckProportion}, we have $$ \frac{2^{n\eta}}{\beta} \leq \frac{\omega_{\gamma(t-s)}(y)}{\omega_{\frac{\gamma}{2}(t-s)}(y)}. $$
    Thus, as desired, we obtain $$\Tilde{C} \mathcal{M} \frac{e^{-\Tilde{c} (\frac{|x-y|^2}{t-s}+\frac{t-s}{s-\mathfrak{T}}})}{\omega_{\gamma(t-s)}(y)} \alpha_{s,s+\frac{t-s}{2}}(y) \leq u(t,x),$$
    with $\Tilde{c}=\Tilde{c}([ \omega  ]_{A_2},n,M,\nu)>0$ and $\Tilde{C}=\Tilde{C}(\gamma,[ \omega  ]_{A_2},n,M,\nu)>0$.
\end{proof}
The following lemma is the final step toward proving Gaussian lower bounds for the generalized fundamental solution, and it is interesting in its own right.
\begin{lem}\label{lem:GLB-SFG}
    There exist two constants $C=C([ \omega  ]_{A_2},n,M,\nu)>0$ and $c=c([ \omega  ]_{A_2},n,M,\nu)>0$ such that
    \begin{equation*}
        \inf_{s<t} \int_{B(y,\sqrt{\gamma(t-s)})} \Gamma(t,x;s,y) \ \mathrm{d}\omega(x) \ \ge C e^{-\frac{c}{\gamma}},
    \end{equation*}
    for all $y\in\mathbb{R}^n$ and $\gamma>0$.
\end{lem}
\begin{proof}
    We follow \cite{aronson1967bounds}. We fix $s,t\in \mathbb{R}$ with $s<t$ and $y\in  \mathbb{R}^n$. For all $\sigma <t$ and $z\in \mathbb{R}^n$, we set 
    $$ v(\sigma,z):= \int_{B(y,\sqrt{\gamma(t-s)})} \Gamma(t,x;\sigma,z) \ \mathrm{d}\omega(x)= \left ( \Tilde{\Gamma}(\sigma,t) \mathbb{1}_{B(y,\sqrt{\gamma(t-s)})} \right )(z). $$
    By definition, $v$ is the unique element in \( L^1_{\mathrm{loc}}((-\infty,t); H^1_\omega(\mathbb{R}^n)) \) such that \( \int_{-\infty}^t \|\nabla_x v(s)\|^2_{2,\omega} \, \mathrm{d}s < \infty \), which is a weak solution to the Cauchy problem
    \[
     \left\{
     \begin{array}{ll}
    \mathcal{H}^\star v = 0 & \text{in } \mathcal{D}'((-\infty,t)\times \mathbb{R}^n), \\
    v(\tau) \to \mathbb{1}_{B(y, \sqrt{\gamma(t-s)})} & \text{in } \mathcal{D'}(\mathbb{R}^n) \text{ as } \tau \to t^-.
    \end{array}
    \right.
    \] 
For all $(\sigma,z)\in \mathbb{R}\times \mathbb{R}^n$, we set 
 \begin{align*}
 \Tilde{A}(\sigma,z)  =\left\{\begin{matrix}
 A^T(\sigma,z) \ \mathrm{if} \ \sigma\leq t, \\ 
 \ I_n \ \ \ \ \  \   \ \  \mathrm{if} \ \sigma> t ,
 \\ 
 \end{matrix}\right. 
\ \ \ \text{and} \ \ 
\Tilde{v}(\sigma,z)    =\left\{\begin{matrix}
v(\sigma,z)  \ \mathrm{if} \ \sigma\leq t, \\ 
 \ 1\ \ \ \ \     \  \mathrm{if} \ \sigma> t ,
 \\ 
 \end{matrix}\right.  
 \end{align*}
 where $A^T$ is the transpose of matrix $A$. By using the $L^2_\omega(\mathbb{R}^n)$-valued continuity of $v$, we see that $\Tilde{v}$ is a nonnegative weak solution to the equation 
 $$ \partial_\sigma \Tilde{v}+\omega^{-1} \mathrm{div}_x( \Tilde{A}(\sigma,\cdot)\nabla_x \Tilde{v})=0 \ \ \text{in} \ \mathcal{D'}(\mathbb{R}\times B(y,\sqrt{\gamma(t-s)})). $$
 Using \eqref{199.5} with $\delta=\frac{1}{2}\sqrt{\gamma(t-s)}$, we obtain
 $$\Tilde{v}(t,y) \leq \Tilde{v}(s,y) e^{C(\frac{16}{\gamma}+1)}.$$
 As $\Tilde{v}(t,y)=1$ and $\Tilde{v}(s,y)=v(s,y)$, we conclude that
 $$ v(s,y)= \int_{B(y,\sqrt{\gamma(t-s)})} \Gamma(t,x;s,y) \ \mathrm{d}\omega(x) \geq e^{-C(\frac{16}{\gamma}+1)}.$$

\end{proof}
We are now ready to derive Gaussian lower bounds for the generalized fundamental solution.
\begin{prop}[Gaussian lower bounds]\label{prop:GLB}
There exist two constants $C=C([ \omega  ]_{A_2},n,M,\nu)>0$ and $c=c([ \omega  ]_{A_2},n,M,\nu)>0$ such that 
\begin{equation}\label{GLB}
     \frac{C}{\omega_{t-s}(y)} e^{-c \frac{\left | x-y \right |^2}{t-s}} \leq \Gamma(t,x;s,y),
\end{equation}
for all $t>s$ and for all $(x,y) \in \mathbb{R}^{2n}$.
\end{prop}
\begin{rem}
    The factor $\frac{1}{\omega_{t-s}(y) }$ appearing in \eqref{GLB} above may be replaced by one of 
    \begin{equation*}
        \frac{1}{\omega_{t-s}(x) }, \ \ \frac{1}{\sqrt{\omega_{t-s}(x)}\sqrt{\omega_{t-s}(y)}}, \ \ \frac{1}{\min(\omega_{t-s}(x) ,\omega_{t-s}(y))},
    \end{equation*}
    and the constants $c$ and $C$ in \eqref{GLB} are replaced by $2c$ and $\Tilde{C}=\Tilde{C}(C,c,D)>0$, respectively. For the proof, see the proof of Lemma \ref{lem:Cruz-Uribe et Rios}.
\end{rem}
\begin{proof}[Proof of Proposition \ref{prop:GLB}]
    We first claim that there exists a constant $C=C([ \omega  ]_{A_2},n,M,\nu)>0$ such that 
    \begin{equation}\label{claim}
        \mathcal{M} := \inf_{t>s} \frac{1}{\ \ \ \left\|  \psi \right\|_{L^1_\omega(B(y,\frac{1}{2}\sqrt{t-s})) \ \ }} \int_{B(y,\sqrt{t-s})} (\Gamma(t,s)\psi)(x)\ \mathrm{d}\omega(x)\ \ge C ,
    \end{equation}
    for all $y \in \mathbb{R}^n$ and all $\psi \in \mathcal{D}(\mathbb{R}^n)$ nonnegative with $\psi(y)=1$. As $\Gamma(\cdot,s)\psi$, for all $s\in \mathbb{R}$ and $\psi \in L^2_\omega(\mathbb{R}^n)$, is a weak solution to the equation $\mathcal{H}u=0$ on $(s,\infty)\times \mathbb{R}^n$, then \eqref{claim} and Lemma \ref{lem:AronsonSerrin} imply that, for all $s\in \mathbb{R}$, $y \in \mathbb{R}^n$ and $\psi \in \mathcal{D}(\mathbb{R}^n)$ nonnegative with $\psi(y)=1$,
    \begin{equation*}
         C \left\|  \psi \right\|_{L^1_\omega(B(y,\frac{1}{2}\sqrt{\frac{t-s'}{2}}))} \frac{e^{-c (\frac{|x-y|^2}{t-s'}+\frac{t-s'}{s'-s})}}{\omega_{t-s'}(y)} \leq (\Gamma(t,s)\psi)(x)= \int_{\mathbb{R}^n} \Gamma(t,x;s,z)\psi(z) \ \mathrm{d}\omega(z),
    \end{equation*}
    for all $t>s \in \mathbb{R}$, $s'\in (s,t)$ and for all $x \in \mathbb{R}^n$. Taking $s'=\frac{t+s}{2}$, we use the Lebesgue differentiation theorem to deduce that 
    \begin{equation}\label{300}
         C \frac{e^{-c \frac{|x-y|^2}{t-s}}}{\omega_{\frac{t-s}{2}}(y)} \leq \Gamma(t,x;s,y),
    \end{equation}
    for all $t>s$ and for all $(x,y) \in \mathbb{R}^{2n}$. Finally, we use \eqref{MuckProportion} to write 
    \begin{equation}\label{301}
        \frac{2^{n\eta}}{\beta} \frac{1}{\omega_{(t-s)}(y)} \leq \frac{1}{\omega_{\frac{1}{2}(t-s)}(y)},
    \end{equation}
    for all $y \in \mathbb{R}^n$. The Gaussian lower bound \eqref{GLB} follows from \eqref{300} and \eqref{301}. 
    
    It remains to prove the claim \eqref{claim}. For $s<t$, $\psi \in \mathcal{D}(\mathbb{R}^n)$ nonnegative with $\psi(y)=1$ and $y \in \mathbb{R}^n$, we have by Fubini's theorem 
    \begin{align*}
        \int_{B(y,\sqrt{t-s})} (\Gamma(t,s)\psi)(x)\ \mathrm{d}\omega(x)&= \int_{B(y,\sqrt{t-s})} \int_{\mathbb{R}^n}\Gamma(t,x;s,z)\psi(z)\ \mathrm{d}\omega(z)\mathrm{d}\omega(x) 
        \\&= \int_{\mathbb{R}^n} \psi(z) \int_{B(y,\sqrt{t-s})} \Gamma(t,x;s,z) \ \mathrm{d}\omega(x)\mathrm{d}\omega(z).
    \end{align*}
    In particular,
    \begin{align*}
        \int_{B(y,\sqrt{t-s})} (\Gamma(t,s)\psi)(x)\ \mathrm{d}\omega(x)&\ge  \int_{B(y,\frac{1}{2}\sqrt{t-s})} \psi(z)\left (  \int_{B(y,\sqrt{t-s})} \Gamma(t,x;s,z) \ \mathrm{d}\omega(x) \right )\mathrm{d}\omega(z)
        \\& \ge \int_{B(y,\frac{1}{2}\sqrt{t-s})} \psi(z)\left (  \int_{B(z,\frac{1}{2}\sqrt{t-s})} \Gamma(t,x;s,z) \ \mathrm{d}\omega(x) \right )\mathrm{d}\omega(z),
    \end{align*}
    as $B(z,\frac{1}{2}\sqrt{t-s}) \subset B(y,\sqrt{t-s})$ for all $z \in B(y,\frac{1}{2}\sqrt{t-s})$.
    Using Lemma \ref{lem:GLB-SFG} with $\gamma=\frac{1}{\sqrt{2}}$, we deduce that 
    \begin{equation*}
        \int_{B(y,\sqrt{t-s})} (\Gamma(t,s)\psi)(x)\ \mathrm{d}\omega(x) \ge C e^{-\sqrt{2} c} \left\|  \psi \right\|_{L^1_\omega(B(y,\frac{1}{2}\sqrt{t-s}))},
    \end{equation*}
    and the claim is proved.
    
\end{proof}

\appendix

\section{Proof of \texorpdfstring{\eqref{sup pour sol faibles}}{Equation (sup pour sol faibles)}}\label{annexe}
The equality \eqref{sup pour sol faibles} follows from this more general lemma.
\begin{lem}\label{Lemma sup supess = supess}
    Let $I \subset \mathbb{R}$ be an interval and $\mathcal{O} \subset \mathbb{R}^n$ an open set and $\mathrm{d}m$ a nonnegative and bounded Borel measure on $\mathcal{O}$. Set $\Omega:= I \times \mathcal{O}$ endowed with the product measure $\mathrm{d}t \otimes \mathrm{d}m$. Then, for any $u \in C(I;L^2(\mathcal{O},\mathrm{d}m))$, we have 
    \begin{equation*}
        M_1:=\sup_{t \in I} \ \supess_{\mathcal{O}} \left|u(t,\cdot) \right|  = \supess_{\Omega} \left|u \right|=:M_2.
    \end{equation*}
\end{lem}
\begin{proof}
Let us first prove that $M_2 \leq M_1$. Fix $\lambda < M_2$. Then, there exists a  Borel set $E_\lambda \subset Q$ such that $(\mathrm{d}t \otimes \mathrm{d}m)(E_\lambda)>0$ and for all $(t,x) \in E_\lambda$, $\left|u(t,x) \right| > \lambda.$ By Fubini's theorem $$(\mathrm{d}t \otimes \mathrm{d}m)(E_\lambda)= \int_I \mathrm{d}m (E_\lambda^t)\mathrm{d}t >0 \ \ \mathrm{with} \ E_\lambda^t := \left\{x \in \mathcal{O} : (t,x) \in E_\lambda \right\}.$$
In particular, there is  $t_0 \in I$ such that $\mathrm{d}m ({E^{t_0}_{\lambda}} ) >0$. Then, $$\lambda \leq \supess_{\mathcal{O}} \left|u(t_0,\cdot) \right| \leq \sup_{I} \ \supess_{\mathcal{O}} \left|u(t,\cdot) \right|=M_1.$$
This is true for all $\lambda < M_2$, therefore $M_2 \leq M_1$. For the converse, using the standard notation $a_+=\max(a,0)$,  for all $a,b\in \mathbb{R}$, we have
\begin{equation}\label{ITR}
    \left|b_+-a_+ \right|\leq \left|b-a \right|.
\end{equation}
We set for all $t\in I$, $f(t):= \int_\mathcal{O} (\left|u(t,x) \right|-\lambda )_+ \mathrm{d}m(x)$. The function $f$ is continuous on $I$. In fact, using \eqref{ITR}, reverse triangular inequality and Cauchy-Schwarz inequality implies that for all $t_1,t_2 \in I$,
\begin{equation*}
    \left|f(t_2)-f(t_1) \right| \leq \sqrt{\mathrm{d}m(\mathcal{O})} \left\|u(t_2,\cdot)-u(t_1,\cdot) \right\|_{L^2(\mathcal{O},\mathrm{d}m)}.
\end{equation*}
Let $\lambda < M_1$. There exists a $t_1 \in I$ such that $\lambda < \supess_{\mathcal{O}} \left|u(t_1,\cdot) \right| $, which is equivalent to the fact that $\int_\mathcal{O} (\left|u(t_1,x) \right|-\lambda )_+ \mathrm{d}m(x) >0$, \textit{i.e.} $f(t_1)>0$. Therefore, by continuity of $f$,
\begin{equation*}
    0 < \int_I f(t) \mathrm{d}t = \iint_{\Omega} (\left|u(t,x) \right|-\lambda )_+ \ (\mathrm{d}t \otimes\mathrm{d}m)(t,x).
\end{equation*}
We deduce that  one has $\left|u(t,x) \right| > \lambda $ on a set of positive $\mathrm{d}t \otimes\mathrm{d}m$ measure. Hence, $M_2 \geq \lambda$. This is true for all $\lambda < M_1$, therefore $M_1\le M_2$.
\end{proof}

\subsubsection*{\textbf{Copyright}}
A CC-BY 4.0 \url{https://creativecommons.org/licenses/by/4.0/} public copyright license has been applied by the authors to the present document and will be applied to all subsequent versions up to the Author Accepted Manuscript arising from this submission.

\bibliographystyle{alpha}
\bibliography{references.bib}

\end{document}